\documentclass[11pt]{article}

\usepackage{amsmath,amsfonts,amsthm,amscd,amssymb,graphicx}

\numberwithin{equation}{section}

\usepackage{hyperref}

\usepackage{color}

\newtheorem{theorem}{Theorem}[section]
\newtheorem{theo}{Theorem}[section]
\newtheorem{lemma}[theorem]{Lemma}

\newtheorem{definition}[theorem]{Definition}

\newtheorem{defi}[theorem]{Definition}

\def\eps{\varepsilon }

\def\beq{\begin{equation}}
\def\eeq{\end{equation}}
\def\bb1{{1\!\!1}}
\def\rit{{\Bbb R}}

\def\nit{{\Bbb N}}

\def\eps{\varepsilon}

\begin{document}

\centerline{\Large \bf Instability of shear layers and Prandtl's boundary layers} 

\bigskip

\centerline{D. Bian\footnote{School of Mathematics and Statistics, Beijing Institute of Technology,  Beijing, China}, 
E. Grenier\footnote{UMPA, CNRS  UMR $5669$, Ecole Normale Sup\'erieure de Lyon, Lyon, France}}



\subsection*{Abstract}


This paper is devoted to the study of the nonlinear instability of shear layers and of Prandtl's boundary layers, for the
incompressible Navier Stokes equations. We prove that generic shear layers are nonlinearly unstable
provided the Reynolds number is large enough, or equivalently provided the viscosity is small enough.
We also prove that, generically,
Prandtl's boundary layer analysis fails for initial data with Sobolev regularity.
In both cases we give an accurate description of the first instability which arises. In some cases
a secondary instability appears, leading to several sublayers and to an unexpected complexity of the flow.


\section{Introduction}


In this paper, we are interested in the inviscid limit of Navier Stokes equations with Dirichlet boundary conditions in the half
plane, and study both the nonlinear stability of shear layers as the viscosity is small, 
and the nonlinear stability of Prandtl's boundary layers. 

\medskip

The classical Navier Stokes equations in the half space $\Omega = \{ y > 0 \}$ read
\beq \label{NS1}
\partial_t u^\nu + (u^\nu \cdot \nabla) u^\nu + \nabla p^\nu - \nu \Delta u^\nu = 0,
\eeq
\beq \label{NS2}
\nabla \cdot u^\nu = 0,
\eeq
\beq \label{NS3}
u^\nu = 0 \quad \hbox{when} \quad y = 0.
\eeq
As $\nu$ goes to $0$, formally, Navier Stokes equations degenerate into the incompressible Euler equations
\beq \label{Eu1}
\partial_t u + (u \cdot \nabla) u + \nabla p  = 0,
\eeq
\beq \label{Eu2}
\nabla \cdot u = 0,
\eeq
\beq \label{Eu3}
u \cdot n = 0 \quad \hbox{when} \quad y = 0,
\eeq
where $n$ is a vector normal to the boundary.

\medskip

Let us first describe the results on shear layers.  
We will consider shear profiles $U^\nu(t,y)$ of the form
\beq \label{heat0}
U^\nu(t,y) = \left(
\begin{array}{cc}
U_s(t,y) \cr
0 
\end{array} \right),
\eeq
where $U_s$ is a given smooth function.
For such flows, Navier Stokes equations reduce to the heat equation
\beq \label{heat1}
\partial_t U_s - \nu \partial_{yy} U_s = 0,
\eeq
together with the boundary condition 
\beq \label{heat2}
U_s(t,0) = 0 .
\eeq
We define
$$
U_+ = \lim_{y \to \infty} U_s(0,y),
$$
and we will assume that $U_+ \ne 0$.

\medskip

As the viscosity goes to $0$, two kinds of instabilities appear, depending on whether the shear flow $U^\nu$ is linearly stable when $\nu = 0$ or not:

\begin{itemize}

\item Some flows are spectrally unstable for linearized Euler equations (namely for $\nu = 0$). According to Rayleigh's criterium, this implies that
they have an inflection point. In this case they remain linearly unstable provided the viscosity is small enough. The speed of growth of
the corresponding linear instabilities is of order $O(1)$, as well as their sizes, in both $x$ and $y$ variables.

 \item Other flows are spectrally stable for linearized Euler equations, for instance concave flows, which do not have inflection points.
 However, even for these flows, provided $\nu$ is small enough, there exists unstable modes for the linearized Navier Stokes equations.
 That is, adding viscosity {\it destabilizes} the flow. These unstable modes, also called Tollmien Schlichting waves (TS waves)
  have been investigated in details in \cite{GGN} and \cite{Bian4}.
 They are characterized by a very slow growth rate, of order $O(\nu^{1/2})$, much slower than the growth rate of the first kind of instabilities.
 Their typical size in $x$ is also very different. It is very large, of order $O(\nu^{-1/4})$ to $O(\nu^{-1/6})$. We thus face long wave instabilities,
 growing very slowly in time.

\end{itemize}

We refer to the Figure $8.8$ of \cite{Drazin} or to the Figure $29$, page $170$, of \cite{Landau}.
Note that for the second kind of flows,  all the physics of the stability  lies in the area $| \alpha | \ll 1$.

\medskip

We already know that, for the first kind of flows, linear instability implies nonlinear instability \cite{GN2}: 
unstable modes grow with a speed $O(1)$, leading to nonlinear perturbations which reach a size $O(1)$ in $L^\infty$. Their typical size, 
both in $x$ and in $y$, is also $O(1)$.

The situation is more complex for the second kind of flows,
 since the corresponding unstable modes grow very slowly, like $\exp( C_0 \nu^{1/2} t)$
only. Let $\phi(t)$ be the size of the nonlinear instability. Then we expect $\phi$ to satisfy a Landau equation of the form
\beq \label{landau}
\dot \phi = \lambda \phi + A | \phi |^2 \phi + O(\phi^5)
\eeq
where $\Re \lambda \approx C_0 \nu^{1/2}$. Note that there is no terms in $\phi^2$ or $\phi^4$ because of symmetries.
The long time evolution of $\phi(t)$ depends on the sign of $A$. By multiplying (\ref{landau}) by $\bar \phi$, we see that if
$\Re A > 0$ then the solution $\phi$, starting with a very small initial data $\phi(0)$, may reach $O(1)$ within times of order
$\nu^{-1/2}$. However if $\Re A < 0$, $\phi(t)$ saturates at $O(\nu^{1/4})$. Preliminary numerical computations, detailed in \cite{Bian2},
show that $\Re A < 0$. We thus expect the nonlinear instability to only reach $O(\nu^{1/4})$, where other phenomena may occur.
In this paper we prove the following result.

\begin{theorem} \label{theofirst2}
Let $U_s(0,y)$ be a smooth   profile such that $U_s(0,0) = 0$, $\partial_y U_s(0,0) \ne 0$, such that $U_s(0,y)$ 
converges exponentially fast to some positive constant $U_+$ and such that $\partial_y^2 U_s(0,y)$ converges exponentially fast to $0$
as $y$ goes to $+ \infty$. Let $U^\nu$ is defined by (\ref{heat0}). 
Let $\theta > 0$ be arbitrarily small.
{Then for any arbitrarily large $N$ and arbitrarily large $s$,
 and for $\nu$ small enough, there exists a vector field $V^\nu$ and a force $F^\nu$,
solutions of Navier Stokes equations, namely
$$
\partial_t V^\nu + (V^\nu \cdot \nabla )  V^\nu + \nabla q^\nu - \nu \Delta V^\nu = F^\nu,
$$
$$
\nabla \cdot V^\nu = 0
$$
with $V^\nu = 0$ at $y = 0$, 
}  a sequence of times $T^\nu$, and a constant $\sigma > 0$ such that
\beq 
\| V^\nu(0,\cdot,\cdot) - U^\nu(0,\cdot) \|_{H^s} \le \nu^N,
\eeq
\beq 
\| F^\nu \|_{L^\infty([0, T^\nu], H^s)} \le \nu^N,
\eeq
\beq \label{theosplit}
\| V^\nu(T^\nu,\cdot,\cdot) - U^\nu(T^\nu,\cdot) \|_{L^\infty} \ge \sigma \nu^{1/4 + \theta} ,
\eeq
\beq 
T^\nu \sim C_0 \nu^{-1/2} \log \nu^{-1} ,
\eeq
and such that $V^\nu(T^\nu,\cdot,\cdot)$ exhibits the scales $\nu^{-1/4}$, $1$ and $\nu^{1/4}$ in the $y$ variable. 
\end{theorem}

These three scales are reminiscent of the "triple deck" which appears in the study of boundary layer separation.
In this paper we construct a nonlinear instability using horizontal waves numbers $\alpha$ of order $\nu^{1/4}$, namely near the
lower marginal stability curve. It is in fact possible to construct instabilities for $\alpha$ of order $\nu^{1/4}$ to $\nu^{1/6}$,
which would lead to different vertical scales.

This theorem describes the first nonlinear step in the transition from a laminar shear layer to a turbulent flow: if the Reynolds
number is large enough, the shear layer is linearly unstable. This linear instability creates a nonlinear instability,
which reaches an amplitude of order at least $\nu^{1/4}$. 
One possibility is that  this two dimensional linear instabilities grow until cubic interactions saturate it, forming "roll like structures".
Note that this nonlinear instability may also later develop secondary instabilities, in particular three dimensional ones.

Let us go back to profiles of the first kind, as discussed previously, namely to profiles which are unstable with
respect to Euler equations.
For such profiles we will prove:

\begin{theorem} \label{theomain}
Let $U_s(0,y)$ be an holomorphic profile which is spectrally unstable for Euler equations. Then
for any arbitrarily large $N$ and $s$,  there exists a sequence of vector fields $V^\nu$ and a sequence of forces $F^\nu$,
solutions of Navier Stokes equations
with forcing term $F^\nu$,  a sequence of times $T^\nu$, and a constant $\sigma > 0$ such that
\beq  \label{d1}
\| V^\nu(0,\cdot,\cdot) - U^\nu(0,\cdot) \|_{H^s} \le \nu^N,
\eeq
\beq \label{d2}
\| F^\nu \|_{L^\infty([0, T^\nu], H^s)} \le \nu^N,
\eeq
\beq \label{d3}
\| V^\nu(T^\nu,\cdot,\cdot) - U^\nu(T^\nu,\cdot) \|_{L^\infty} \ge \sigma  > 0,
\eeq
\beq \label{d4}
T^\nu \sim C_0 \log \nu^{-1} ,
\eeq
and such that  $V^\nu(T^\nu,\cdot,\cdot)$ exhibits the scales $1$, $\nu^{1/2}$, $\nu^{3/8}$ and $\nu^{5/8}$ in the $y$ variable.

If $U_s(0,y)$ is $C^\infty$ but not analytic, the same result holds true, provided (\ref{d3}) is replaced by
\beq \label{d3bis}
\| V^\nu(T^\nu,\cdot,\cdot) - U^\nu(T^\nu,\cdot) \|_{L^\infty} \ge \nu^\theta  > 0,
\eeq
where $\theta > 0$ may be chosen arbitrarily small.

\end{theorem}

The existence of a sequence of solutions satisfying (\ref{d1})-(\ref{d4})
has already been obtained in \cite{GN2}. Theorem \ref{theomain} 
 improves this result by describing a secondary instability, which leads to two other sublayers, of
sizes $\nu^{3/8}$ and $\nu^{5/8}$.
The structure of the flow at $T^\nu$ is thus much more complex than expected, and, up to the best of our knowledge, not 
described in the classical physics literature \cite{Drazin,Landau}. 

This theorem describes the successive development of two instabilities.
First, the shear layer $U^\nu(0,\cdot)$ is linearly unstable, with an unstable mode $u_{lin}$ which displays two sizes, namely $1$ and $\nu^{1/2}$.
As the instability grows, its "sublayer", of size $\nu^{1/2}$, increases.
The Reynolds number of this sublayer, defined as the product of its typical size by its typical velocity divided by the viscosity,
is of magnitude of $\nu^{1/2} \times 1 / \nu = \nu^{-1/2}$, namely goes to infinity as $\nu \to 0$. As any boundary layer  becomes
unstable provided its Reynolds number is large enough, this sublayer itself becomes unstable.
Secondary instabilities appear in this sublayer, which themselves display  smaller scales, including
a "subsublayer" of order $\nu^{5/8}$, and a "recirculation layer" of order $\nu^{3/8}$.

\medskip

Let us now turn to Prandtl boundary layers.
Despite many efforts, the question of the convergence of solutions of Navier Stokes equations to solutions of Euler equations
is widely open. A classical approach is to introduce the so called Prandtl's boundary layer, of size $O(\sqrt{\nu})$. This leads
to look for $u^\nu$ under the form 
\beq \label{Prandtl}
u^\nu(t,x,y) = u^{Euler}(t,x,y) + u^{Prandtl}(t,x,\nu^{-1/2} y) + O(\sqrt{\nu})_{L^\infty},
\eeq
where $u^{Euler}$ solves Euler equations and $u^{Prandtl}$ is a corrector, located near the boundary, solution of so called
Prandtl's equations.

The expansion (\ref{Prandtl}) has been justified in small time for initial data with analytic regularity
 by R.E. Caflish and M. Samartino \cite{Caflish} and has later been extended to data with Gevrey regularity, or for initial data
with vorticity supported away from the boundary \cite{Maekawa}. These results show that the formal analysis of Prandtl is valid in particular cases.

In the analytic case for instance, for small $u^\nu$, there exists two sequences of analytic functions 
$u^{Euler,\nu}(t,x,y)$ and $u^{Prandtl,\nu}(t,x,y)$ such that
\beq \label{Prandtl2}
u^\nu(t,x,y) = u^{Euler,\nu}(t,x,y) + u^{Prandtl,\nu} \Bigl( t, x, {y \over \sqrt{\nu}} \Bigr) .
\eeq
Moreover, $u^{Euler,\nu}$ and $u^{Prandtl,\nu}$ may be expanded in series  in $\nu$, with in particular
\beq \label{Prandtl3}
u^{Euler,\nu}(t,x,y) = u^{Euler}(t,x,y) + O(\sqrt{\nu}),
\eeq
and similarly
\beq \label{Prandtl4}
u^{Prandtl,\nu}(t,x,Y) = u^{Prandtl}(t,x,Y) + O(\sqrt{\nu}),
\eeq
where $Y = y / \sqrt{\nu}$ is the rescaled vertical variable with respect to boundary layer's size.
In this article, we prove that Prandtl's Ansatz is {\it false} in the case of solutions with Sobolev regularity.
A first result in this direction has been done in \cite{GN2}, where it is proven that Prandtl's analysis fails 
for {\it some particular} initial data.
 Theorem \ref{theomain2} improves this result by describing secondary instabilities. 

\begin{theorem} \label{theomain2}
Let $u(t,x,y)$ be a smooth solution to Navier Stokes equations, which follows Prandtl's Ansatz, namely of the form (\ref{Prandtl2}),
(\ref{Prandtl3}) and (\ref{Prandtl4}).
Let us assume that there exists some $t_0 \ge 0$ and some $x_0$ such that 
$$
Y \to u^{Euler}(t_0,x_0,0) + u^{Prandtl}(t_0,x_0,Y), 
$$
is  spectrally unstable for Euler equations.  Then
for any $N$ and $s$ arbitrarily large and any $\theta > 0$ arbitrarily small,  there exists a sequence of solutions $u^\nu$ of Navier Stokes equations
with forcing term $f^\nu$,   a sequence of times $T^\nu$, and a constant $\sigma > 0$ such that,
\beq
 \forall t \le t_0, \qquad u^\nu(t ,\cdot,\cdot) =  u(t,\cdot,\cdot) \qquad \hbox{and} \qquad f^\nu(t,\cdot,\cdot) = 0,
\eeq
\beq 
\| f^\nu \|_{L^\infty([t_0,t_0 +  T^\nu], H^s)} \le \nu^N,
\eeq
\beq
\| u^\nu(t_0 + T^\nu,\cdot,\cdot) - u(t_0 + T^\nu,\cdot,\cdot) \|_{L^\infty} \ge \sigma \nu^\theta,
\eeq
\beq 
T^\nu \sim \nu^{1/2} C_0 \log \nu^{-1} \to 0,
\eeq
and such that  $u^\nu(t_0 + T^\nu,\cdot,\cdot)$ exhibits the scales $1$, $\nu^{1/2}$, $\nu^{3/4}$,
$\nu^{11/16}$ and $\nu^{13/16}$ in the $y$ variable.
\end{theorem}

In other word, in this case, Prandtl's layer is linearly unstable at time $t_0$, leading to a nonlinear instability which creates
a "viscous sublayer" of size $\nu^{3/4}$. This sublayer becomes itself unstable, and create two other sublayers, of
size $\nu^{11/16}$ and $\nu^{13/16}$.
By repeating the arguments of \cite{GN2}, if $u$ is analytic, it is even possible to take $\theta = 0$, but we will not detail this point here.
Note again that $\nu^{11/16}$ and $\nu^{13/16}$ are new sizes which, up to the best of our knowledge, has never been discussed before. 

 The following Theorem states that Prandtl's Ansatz is {\it generically false} in the case of Sobolev regularity.

\begin{theorem} \label{theomain4}
Let $u(t,x,y)$ be a smooth solution to Navier Stokes equations, which follows Prandtl's Ansatz, namely of the form (\ref{Prandtl2}),
(\ref{Prandtl3}) and (\ref{Prandtl4}).
Let us assume that there exists  $x_0$ such that 
\beq \label{assumpderiv}
 \partial_Y u^{Prandtl}(0,x_0,0) \ne 0.
\eeq
Let $\theta > 0$ be arbitrarily small.
Then
for any $N$ and $s$ arbitrarily large,   there exists a sequence of solutions $u^\nu$ of Navier Stokes equations
with forcing term $f^\nu$,  a sequence of times $T^\nu$, and a constant $\sigma > 0$ such that
\beq
\| u^\nu(0,\cdot,\cdot) - u(0,\cdot,\cdot) \|_{H^s} \le \nu^N ,
\eeq
\beq 
\| f^\nu \|_{L^\infty([0, T^\nu], H^s)} \le \nu^N,
\eeq
\beq
\| u^\nu(T^\nu,\cdot,\cdot) - u(T^\nu,\cdot,\cdot) \|_{L^\infty} \ge \sigma \nu^{1/4 + \theta} ,
\eeq
\beq 
T^\nu \sim \nu^{1/4} C_0 \log \nu^{-1} \to 0,
\eeq
and such that $u^\nu(T^\nu,\cdot,\cdot)$ exhibits the scales $1$, $\nu^{1/2}$, $\nu^{3/8}$ and $\nu^{5/8}$ in the $y$ variable.
\end{theorem}

Thus Prandlt's Ansatz is  {\it generically false} in Sobolev spaces and only holds in small times in spaces with analytic or Gevrey regularity. 
Note that (\ref{assumpderiv}) may be replaced 
by $\partial_Y u^{Prandtl}(0,x,0) = 0$ for any $x$ and $\partial_Y^2 u^{Prandtl}(0,x_0,0) \ne 0$ for some $x_0$, up to additional
technicalities. This result can also be extended to the classical Blasius boundary layer \cite{Blasius3} in the whole space.

 Theorem \ref{theomain4}  describes the first nonlinear step in the evolution of a laminar flow (which follows Prandtl's assumption) towards
a more complex, or even turbulent one: an instability appears in Prandtl's layer, leading to a smaller scale.  A possibility is that
"roll like structures" form in this layer, structures which can then develop further instabilities.

\medskip

The rest of the paper is organized as follows.  In the second section
 we recall previous results on Orr Sommerfeld equations, design various functional spaces adapted
to them, and study in detail their inverse, both away and close to its eigenvalues. In the third section we prove a general instability theorem
which allows us to localise instabilities. In the fourth section we investigate the case of "fast instabilities"
and in the last part we prove the various theorems stated in this introduction.


\subsubsection*{Notations}


In all this article, "$\hbox{c.c.}$" is a shorthand for "complex conjugate".
Moreover, for any $\alpha \in \rit$, we define 
$$
\Delta_\alpha = \partial_y^2 - \alpha^2, \qquad \nabla_\alpha = (- i \alpha, \partial_y) .
$$
For any real number $t$, we define $\langle t \rangle$ by
$$
\langle t \rangle = 1 + |t| .
$$
We say that $f \lesssim g$ is there exists a constant $C$ such that $|f| \le C |g|$.

We draw the attention of the reader upon a change of definition of the horizontal wave number $\alpha$, which is rescaled
at (\ref{defitildealpha}).


\section{Preliminaries  }



\subsection{Orr Sommerfeld equations   }


Let us first introduce the classical Orr Sommerfeld equations. Let $U(y) = (U_s(y),0)$ be a time independent shear layer.
We define the linearised Navier Stokes operator $L_{NS}[U]$ by
$$
L_{NS}[U] \, v =  P \Bigl( (U \cdot \nabla) v +  (v \cdot \nabla) U - \nu \Delta v  \Bigr)
$$
together with $\nabla \cdot v = 0$ and Dirichlet boundary condition $v = 0$ on $y = 0$, 
where $P$ is Leray's projection on divergence free vector fields.
We will also use the linearised Navier Stokes operator in vorticity formulation $L_{NS}^\omega[U]$, defined by
$$
L_{NS}^\omega [U] \, \omega = (U \cdot \nabla) \omega + (v \cdot \nabla) \Omega-\nu \Delta \omega
$$
where $\omega = \nabla \times u$ is the vorticity of $v$ and $\Omega$ the vorticity of the shear layer $U$.

We recall that the solution $v(t)$ of linearised Navier Stokes equations 
is given by  Dunford's formula
\beq \label{contour1}
v(t,x,y) = {1 \over 2 i \pi} \int_\Gamma e^{\lambda t} \Bigl( L_{NS}[U] + \lambda \Bigr)^{-1} v(0,\cdot,\cdot) \; d\lambda
\eeq
in which $\Gamma$ is a contour on the "right" of the spectrum.
 We are therefore led to study the resolvent of $L_{NS}[U]$, namely to
study the equation
\beq \label{resolvant}
\Bigl( L_{NS}[U] + \lambda \Bigr) v = f,
\eeq
where $f(x,y)$ is a given forcing term and $\lambda$ a complex number.
To take advantage of the divergence free condition, we introduce the stream function $\psi$ of $v$ and take its Fourier
transform in $x$, with dual variable $\alpha$, which leads to look for solutions of (\ref{resolvant}) of the form
$$
v = \nabla^\perp \Bigl( e^{i \alpha x  } \psi(y) \Bigr) .
$$
According to  traditional notations \cite{Reid}, we introduce the complex number $c$, defined by  
$$
\lambda = - i \alpha c .
$$
We also take the Fourier transform of the forcing term $f$
$$
f(x,y) = \Bigl( f_1(y),f_2(y) \Bigr) e^{i \alpha x  } .
$$
Taking the curl of (\ref{resolvant}) and dividing by $i \alpha$, we get the classical Orr Sommerfeld equation, namely
\beq \label{Orrmod0}
Orr_{\alpha,c,\nu}(\psi) =  (U_s - c)  (\partial_y^2 - \alpha^2) \psi - U_s''  \psi  - {\nu \over i \alpha}  (\partial_y^2 - \alpha^2)^2 \psi 
=  i {\nabla \times f \over \alpha}
\eeq
where
$$
\nabla \times (f_1,f_2) = i \alpha f_2 - \partial_y f_1 .
$$
Note the $i \alpha^{-1}$ factor in the right hand side of (\ref{Orrmod0}), which appears when we go from $L_{NS}^\omega$
to the Orr Sommerfeld equations.
Note that these equations are simply the resolvent of linearised Navier Stokes equations in vorticity formulation, and divided
by $i \alpha$.

For a given $c$, close to $0$, we define the critical layer $y_c$ to be the solution of
$$
U(y_c) = c .
$$
In the whole paper, both $c$ and $y_c$ will be of order $\nu^{1/4}$. We say that $y$ is in the critical layer if $y$ is of order $\nu^{1/4}$.

The Orr Sommerfeld equation has been fully studied in \cite{Bian4,Zhang}. 
It turns out that any shear layer $(U_s(y),0)$ with $U_s'(0) \ne 0$ is  linearly unstable in the long wave regime.
We refer to \cite{Bian4} for more details in the case
of holomorphic and concave profiles, and to \cite{Zhang} for the general case
(see also \cite{zhang2}  and \cite{yang-zhang} for the compressible case).  

More precisely, there exists two functions $\alpha_\pm(\nu)$, such that
if $\alpha_-(\nu) < | \alpha | < \alpha_+(\nu)$, then there exists only one eigenvalue $c(\alpha,\nu)$
of $L_{NS}$ with $\Im c(\alpha,\nu) > 0$. 
Moreover $\alpha_-(\nu) \sim C_- \nu^{1/4}$ and $\alpha_+(\nu) \sim C_+ \nu^{1/6}$ for some constants $C_-$ and $C_+$, as $\nu \to 0$.

We will focus on the lower branch of instability, namely on instabilities arising when $\alpha$ is of order $\alpha_0 \nu^{1/4}$ 
for some fixed, positive and large enough $\alpha_0$.
In this case, $\Re c(\alpha,\nu)$ and $\Im c(\alpha,\nu)$ are both of order $\nu^{1/4}$.
We refer to section \ref{linear} for more details on $c(\alpha,\nu)$ and the associated eigenmode $\psi_{\alpha,\nu}$.

Let us first consider the equation 
\beq \label{Orrzero}
Orr_{c,\alpha,\nu}(\psi) = 0,
\eeq
without taking care of the boundary conditions at $0$ and at infinity.
As  this equation is a fourth order ordinary differential equation, it has four independent solutions. It can be proved
that two of them have a "slow" behavior, one going to $0$ as $y$ goes to $+ \infty$ and the other one being unbounded
at infinity. We will call
$\phi_{s,-}$ and $\phi_{s,+}$ these two solutions. There also exist two solutions with a "fast" behavior, one going quickly to $0$ as
$y$ goes to $+ \infty$ and the other going to infinity. We will call them $\phi_{f,-}$ and $\phi_{f,+}$. 

It turns out that $\phi_{f,-}$ and $\phi_{f,+}$ can be expressed as primitives of the classical Airy functions \cite{Bian4,Zhang},
whereas $\phi_{s,-}$ and $\phi_{s,+}$ are perturbations of solutions of the Rayleigh equation, obtained by setting $\nu = 0$ in 
Orr Sommerfeld equations, namely
\beq \label{Rayleigh}
(U_s - c) (\partial_y^2 - \alpha^2) \psi - U_s'' \psi = 0 .
\eeq
To recall the construction of the Green function which will be used in section \ref{defiGreen}, 
we need to describe in details the fast mode $\phi_{f,-}$. 
First we can normalise it in such a way that 
\beq \label{normalisationphi}
\phi_{f,-}(0) = 1.
\eeq
Second, in the case of holomorphic profiles $U_s(y)$,
 for $y$ close to $0$, following Langer's transformation, we introduce the function $g(y)$, defined by
\beq \label{explicitg}
g(y) = y_c + \Bigl( {3 \over 2 \sqrt{U_s'(y_c)}} \int_{y_c}^y  \sqrt{U_s(z) - c} \, dz \Bigr)^{2/3},
\eeq 
and the complex number $\gamma$, defined by
\beq \label{defigamma}
\gamma = \Bigl(  {i \alpha U_s'(y_c) \over \nu} \Bigr)^{1/3} .
\eeq
Note that, in the sequel, $\gamma$ will be of order $\nu^{-1/4}$.
Then, if $Ai$ denotes the classical Airy function, we have
$$
\phi_{f,-}(y) = Ai_a(y) \Bigl[ 1 + O(\gamma^2) \Bigr], 
$$
where
$$
Ai_a(y) = Ai \Bigl( \gamma (g(y) - y_c) \Bigr) .
$$
The analysis is similar if $U_s(y)$ is smooth but not analytic (see \cite{Zhang}).
We further define $\mu(y)$, the "speed of variation" of $\phi_{f,-}(y)$, by
\beq \label{defimu}
\mu(y) = {\partial_y Ai_a(y) \over Ai_a(y) } .
\eeq
As proved in \cite{Bian4}, $\mu(y) \to C \gamma^{3/2}$ when $\gamma (y - y_c) \to + \infty$, namely when $y$ is away from the
critical layer, and $\mu(y)$ is of order $\gamma$ if $\gamma (y - y_c)$ is of order $1$, namely when $y_c$ is in the critical layer.

We will also use that $\phi_{f,-}$ and $\phi_{f,+}$ have a fast exponential behaviour. Namely, for any arbitrarily large
positive constant $C_1$, provided $\nu$ is small enough,
\beq \label{decrease}
\Bigl| {\phi_{f,-}(y) \over \phi_{f,-}(x) } \Bigr| \lesssim e^{- C_1 |y - x|}
\eeq
provided $x \le y$, and similarly for $x \ge y$ and for $\phi_{f,+}$.

We will also need  some results about the adjoint of Orr Sommerfeld equation, introduced in \cite{Bian1,Bian4} and detailed in \cite{Bian6}.
Let us first recall the computation of the adjoint of $L_{NS}$. Let $V_1 = (u_1,v_1)$ and $V_2= (u_2,v_2)$
be two divergence free vector fields with corresponding stream functions $\psi_1(x,y)$ and $\psi_2(x,y)$ and corresponding vorticities
$\omega_1(x,y)$ and $\omega_2(x,y)$.
Then
\begin{equation*} \begin{split}
\int_\Omega L_{NS} V_1 \cdot \bar V_2 \, dx \, dy &=
\int_\Omega L_{NS}^\omega \omega_1 \cdot \bar \psi_2 \, dx \, dy
\\ &= \int_\Omega \Bigl[ U i \alpha \omega_1 + v_1 \partial_y \Omega - \nu \Delta_\alpha \omega_1 \Bigr] \cdot \bar \psi_2 \, dx \, dy
\\ &=   \int_\Omega \Bigl[ - U i \alpha \Delta_\alpha \psi_1 
- i \alpha \psi_1 \partial_y \Omega + \nu \Delta_\alpha^2 \psi \Bigr] \cdot \bar \psi_2 \, dx \, dy
\\ &= \int_\Omega \psi_1 \cdot 
\Bigl[ - i \alpha \Delta_\alpha (U \bar \psi_2) - i \alpha \bar \psi_2 \partial_y \Omega + \nu \Delta_\alpha^2 \bar \psi_2 \Bigr] 
\, dx \, dy .
\end{split} \end{equation*}
This leads to the following definition of the "adjoint Orr Sommerfeld equation", which, 
following the usual definition in physics books (in particular \cite{Reid}, page $254$), reads
\beq \label{Orrad}
Orr^t_{\alpha,c,\nu}(\psi) :=  (\partial_y^2 - \alpha^2) (U_s -  c)   \psi 
- U_s''  \psi  
+ { \nu \over i \alpha}   (\partial_y^2 - \alpha^2)^2 \psi,
\eeq
with boundary conditions $\psi(0) = \partial_y \psi(0) = 0$. Adjoint Orr Sommerfeld equations are the resolvent of the
adjoint of the linearised Navier Stokes equations in vorticity formulation, after a division by $i \alpha$.

If $c$ is an eigenvalue of $Orr_{\alpha,c,\nu}$ (more precisely an eigenvalue of $L_{NS}^\omega$), 
then $c$ is also an eigenvalue of $Orr^t_{\alpha,c,\nu}$ (more precisely $(L_{NS}^\omega)^t$).
Moreover, if $\phi_1$ is an eigenvector of $Orr_{c,\alpha,\nu}$ with corresponding eigenvalue $c_1$,
and $\phi_2^t$ an eigenvector of $Orr_{c,\alpha,\nu}^t$ with corresponding eigenvalue $c_2$, then 
$\phi_1$ and $\phi_2^t$ are orthogonal in the sense
\beq \label{ortho}
\int_{\rit^+} \nabla^\perp \phi_1(y) \cdot \nabla^\perp \bar \phi_2^t (y) \, dy = \delta_{c_1 = c_2} .
\eeq
In other terms, $Orr_{\alpha,c,\nu}$ and $Orr_{\alpha,c,\nu}^t$ are "adjoint" with respect to the $L^2$ scalar product of the
corresponding velocities.

Following \cite{Bian4}, the associated eigenvector $\phi^t(y)$ is bounded by 
\beq \label{boundpsit}
| \phi^t (y) | \lesssim \nu^{-1/4} e^{- C \nu^{-1/4}}  + e^{-C y} + e^{- \alpha y}
\eeq
for some positive constant $C$.


\subsection{Linear instability \label{linear}}


In \cite{Bian4} (in the analytic case) and \cite{Zhang} (in the $C^\infty$ case), 
it is proven that any shear layer $U_s(y)$ such that $U_s'(0) \ne 0$ is linearly unstable.
More precisely,  if $\alpha_0 >0$ is large enough
and $\nu$ small enough, then, for  $\alpha = \alpha_0 \nu^{1/4}$,
 there exists an unstable mode $\psi_\alpha^{unstable}(t,x,y)$,
which is of the form
\beq \label{formunstable}
\psi_\alpha^{unstable}(t,x,y) = e^{i \alpha (x - c(\alpha,\nu) t)} \widetilde \psi^{unstable}_\alpha(y) ,
\eeq
where $c(\alpha,\nu)$ is of order $O(\nu^{1/4})$, and where, for every $n \ge 0$ and $p \ge 0$,
\beq \label{formunstable2}
| \partial_y^n \partial_\alpha^p \widetilde \psi^{unstable}_{\alpha}(y) | \lesssim 
| \alpha |^n e^{- C_0 | \alpha | y} + e^{- C_0 y} + \nu^{-(n - 1) /4} e^{- C_0 y / \nu^{1/4}} ,
\eeq
where $C_0$ is a small positive constant.

An important point is that the corresponding vorticity 
$$
\widetilde \omega_{\alpha}^{unstable}(y)  =- (\partial_y^2 - \alpha^2) \widetilde \psi_\alpha^{unstable}(y)
$$ 
has a better behavior at infinity than $\widetilde \psi_\alpha^{unstable}$, namely, for every $n > 0$,
\beq
| \partial_y^n  \partial_\alpha^p \widetilde \omega_{\alpha}^{unstable}(y) | \lesssim  
 e^{- C_0 y} + \nu^{-(n + 1) /4} e^{- C_0 y / \nu^{1/4}} .
\eeq
Note that there are three scales in $y$: $O(\nu^{-1/4})$, corresponding to a large scale irrotational recirculation area, $O(1)$, the typical
size of the shear layer, and $O(\nu^{1/4})$, size of the so called critical layer, which, in the case of the lower marginally stable
branch, is stuck to the boundary. Moreover, $\psi_\alpha^{unstable}(t,x,y)$ evolves over times of order $\alpha^{-1} c(\alpha,\nu)^{-1} \sim \nu^{-1/2}$.

The dispersion relation $c(\alpha,\nu)$ may easily be computed and
 numerical evidences \cite{Bian1} indicate that
there exists  $\alpha_0(\nu) > C_-$, converging to some positive constant $\alpha_0 > C_-$ such that 
\beq \label{nearmax}
\Re \lambda \Bigl( \alpha_0(\nu) \nu^{1/4},\nu \Bigr) =  \max_\alpha \Re \lambda(\alpha,\nu) ,
\eeq
provided $\nu$ is small enough. To simplify the notation, we will omit the $\nu$ dependency of $\alpha_0(\nu)$.

We also note that, for any $\alpha$, using the explicit form of Orr Sommerfeld equations,
\beq \label{symmetrylambda}
\lambda(-\alpha,\nu) = \bar \lambda(\alpha,\nu).
\eeq
In particular, $\lambda(\alpha,\nu)$ and $\lambda(-\alpha,\nu)$ have the same real part.

As, from now on, we consider $\alpha$ of order $\nu^{1/4}$, we rescale $\alpha$ and introduce
\beq \label{defitildealpha}
\tilde \alpha = {\alpha \over \nu^{1/4}} .
\eeq
We further rescale $\alpha_0(\nu)$ by a factor $\nu^{1/4}$ and define $\tilde \alpha_0(\nu) = \alpha_0(\nu) / \nu^{1/4}$.
We also rescale $\lambda(\alpha,\nu)$, which is of the form
$$
\lambda(\alpha,\nu) = \nu^{1/2}  \tilde \lambda(\tilde \alpha,\nu) = - i \nu^{1/4} \tilde \alpha c(\tilde \alpha,\nu),
$$
where $\tilde \lambda(\tilde \alpha,\nu)$ is a smooth function of $\tilde \alpha$ and $\nu$. 
We will not rescale $c(\tilde \alpha,\nu)$ which appears in the definition of the Orr Sommerfeld equation.
Note that $c(\tilde \alpha,\nu)$ is of order $\nu^{1/4}$.

\medskip

{\it From now on we drop the tildes on $\tilde \alpha$, $\tilde \alpha_0$ and $\tilde \lambda$.
For the sake of simplicity, we omit the $\nu$ in the notation of $\tilde \lambda$ and $\alpha_0$. 
}

\medskip

\noindent With this  rescaling, (\ref{formunstable}) becomes
\beq \label{formunstablebbis}
\psi_\alpha^{unstable}(t,x,y) = e^{i \alpha  \nu^{1/4}  x + \nu^{1/2} \lambda(\alpha) t} \widetilde \psi^{unstable}_\alpha(y) ,
\eeq
and  (\ref{formunstable2}) becomes: for any integers $n \ge 0$ and $p \ge 0$,
\beq \label{formunstable3}
| \partial_y^n \partial_\alpha^p \widetilde \psi^{unstable}_{\alpha}(y) | \lesssim 
\nu^{n/4}  e^{- C_0 \nu^{1/4}  y} + e^{- C_0 y} + \nu^{-(n - 1) /4} e^{- C_0 y / \nu^{1/4}} ,
\eeq
provided $C_0$ is small enough.

 We  "localise" the instability in the $x$ variable by building a "wave package".  Let $\chi$ be a non negative smooth function,
 compactly supported in $[-1,+1]$, with unit integral. Let $\beta > 0$ be a small positive number. 
 We define the "wave package" $\Psi_{lin}$ by
\beq \label{local}
\Psi_{lin}(t,x,y) = \nu^{-\beta} \int_{\alpha_0 - \nu^\beta}^{\alpha_0 + \nu^\beta}
 \chi \Bigl( {\alpha - \alpha_0 \over \nu^\beta} \Bigr) \widetilde \psi_{\alpha}^{unstable}(y) 
 e^{i \alpha \nu^{1/4} x +  \nu^{1/2}  \lambda(\alpha) t} \, d\alpha ,
\eeq
with corresponding vorticity $\omega_{lin}(t,x,y)$.

As $\tilde \psi_{\alpha}^{unstable}$ is smooth in all its variables, we can use classical stationary phase arguments to study
the long time behavior of $\Psi_{lin}$. The phase is
$$
\theta(t,x) = i \nu^{1/4} \alpha x + i \nu^{1/2} \lambda(\alpha,\nu) t 
$$
and is stationary when
$$
\partial_\alpha \theta = i \nu^{1/4} x + i \nu^{1/2}  \lambda'(\alpha) t = 0.
$$
The real part of $\partial_\alpha \theta$ vanishes at $\alpha_0$, where $\Re \lambda$ reaches its maximum (\ref{nearmax}).
The imaginary part of $\partial_\alpha \theta$ vanishes if $x / t = - \nu^{1/4} \Im \lambda'(\alpha_0)$.
Thus, $\Psi_{lin}$ grows like $e^{\nu^{1/2} \Re  \lambda(\alpha_0)   t} / \sqrt{\nu^{1/2} t}$,  namely over time scales of order $\nu^{-1/2}$,
and "travels" with the "group velocity" $c_\sigma$, where
\beq \label{groupvelocity}
c_\sigma = -  \nu^{1/4} \Im \lambda'(\alpha_0) .
\eeq
It rapidly decays  away from $c_\sigma t$, with a typical horizontal scale $\nu^{-\beta - 1/4} \gg \nu^{-1/4}$.
Note that the "localized" instability $\Psi_{lin}$ travels in $x$ with a speed $c_\sigma$ of order $O(\nu^{1/4})$, 
which refers to the notion of "convective instability". Note also that, during the characteristic evolution time $O(\nu^{-1/2})$, $\Psi_{lin}$
travels over a length of order $O(\nu^{-1/4})$, namely of order of the characteristic size in $x$ of the perturbation.
In particular, 
\beq \label{maxPsilin}
\max_x | \Psi_{lin}(t,x) | \sim C {e^{\nu^{1/2} \Re \lambda( \alpha_0) t} \over  \langle \sqrt{\nu^{1/2} t} \rangle } 
\eeq
for some constant $C$.

However, as we are only interested on times of order $\nu^{-1/2} \log \nu^{-1}$, the wave packet only travels over lengths of order
$\nu^{-1/4} \log \nu^{-1}$, which may be neglected with respect to the $\nu^{-\beta - 1/4}$ decay.
In fact, a crude bound on $\Psi_{lin}$ will be sufficient and may be obtained by rewriting it under the form
\beq \label{Psibis}
\Psi_{lin}(t,x,y) = \int_{\alpha_0 - \nu^\beta}^{\alpha_0 + \nu^\beta}
\phi(\alpha,t) e^{i \alpha \nu^{1/4} x} \widetilde \psi_{\alpha}^{unstable}(y) \, d\alpha,
\eeq
where
$$
\phi(\alpha,t) = \nu^{-\beta} \chi \Bigl( {\alpha - \alpha_0  \over \nu^\beta} \Bigr) e^{\nu^{1/2} \lambda(\alpha) t} .
$$
Now, for $t \lesssim \nu^{-1/2} \log \nu^{-1}$ and for any $N \ge 0$,   
$$
| \partial_\alpha^N \phi(\alpha,t) | \lesssim \nu^{-\beta} \nu^{-\beta N} e^{\nu^{1/2} \Re \lambda(\alpha_0) t}.
$$
Thus, integrating (\ref{Psibis}) by parts in $\alpha$, we get, for any integers $N \ge 0$ and $n \ge 0$, using (\ref{formunstable3}),
\beq \label{Psiter}
| \partial_y^n \Psi_{lin}(t,x,y) | \lesssim {
\nu^{n/4} e^{- C_0    \nu^{1/4} y } + e^{- C_0 y} + \nu^{- (n - 1) /4} e^{- C_0 y / \nu^{1/4}} 
\over 1 + | \nu^{\beta + 1/4} x |^N}  e^{\nu^{1/2} \Re \lambda(\alpha_0) t}  ,
\eeq
for some small constant $C_0$ depending on $N$ and $n$.
This bound is less precise than the stationary phase analysis, but appears to be sufficient in our case.
Note that we have introduced a new spatial size in $x$, of order $\nu^{-\beta - 1/4}$.


\subsection{Functional spaces}


The previous paragraph motivates the introduction of the spaces $X^{n,p}$, $X^n_\omega$ and $Y^n$
to describe the regularity of stream functions and vorticities. 

\begin{defi}
Let $C_0 > 0$.
Let $n$ and $p$ be two integers.
We say that a function $f$ belongs to $X^{n,p}$
if there exists a constant $C \ge 0$ such that for every $0 \le j \le n$ and every $y \ge 0$,
\beq  
| \partial_y^j f(y) | \le C \Bigl[ \nu^{j/4}  e^{- C_0 \nu^{1/4}  y} + e^{- C_0 y}  + \nu^{-(j+p)/4} e^{- C_0 y / \nu^{1/4}}  \Bigr]
\eeq
and we denote  by $\| f \|_{X^{n,p}}$ the best possible constant $C$.

We moreover say that a function $f$ belongs to $X^n_\omega$
if there exists $C$ such that for every $0 \le j \le n$ and every $y \ge 0$,
\beq
| \partial_y^j f(y) | \le C \Bigl[  e^{- C_0 y}  + \nu^{-(j+1)/4} e^{- C_0 y / \nu^{1/4}}  \Bigr]
\eeq
and we denote by $\| f \|_{X^n_\omega}$ the best possible constant $C$.

We also  define the space $Y^n$ to be the space of  functions $\psi$ such that 
$$
\| f \|_{Y^n} = \| f \|_{X^{n,-1}} + \| (\partial_y^2 - \alpha^2) f \|_{X^{n-2}_\omega} < + \infty.
$$
\end{defi}

In particular, we note that, provided $C_0$ is small enough,
 for any $n \ge 0$, $\psi_\alpha^{unstable} \in X^{n,-1}$, $\omega_\alpha^{unstable} \in X^n_\omega$,
and thus $\psi_\alpha^{unstable} \in Y^n$.

Now if $\psi^1$ and $\psi^2$ are two stream functions we define the transport term ${\cal Q}(\psi^1,\psi^2)$ by
$$
{\cal Q}(\psi^1,\psi^2) = - \nabla_\alpha^\perp \psi^1 \cdot \nabla_\alpha \Delta_\alpha \psi^2 
= (v_1 \cdot \nabla) \omega_2,
$$
where $v_1$ is the velocity related to $\psi_1$ and $\omega_2$ the vorticity related to $\psi_2$.

\begin{lemma}
If $\psi \in Y^n$, then the related velocities $u_1 = - \partial_y \psi$ and $u_2 = i \alpha \nu^{1/4} \psi$,
and the related vorticity $\omega = - \Delta_\alpha \psi$ satisfy
\beq \label{u1}
\| u_1 \|_{X^{n-1,0}} \lesssim \| \psi \|_{Y^n}, \qquad
\|  u_2 \|_{X^{n,-1}} \lesssim  \nu^{1/4} | \alpha | \, \| \psi \|_{Y^n} ,
\qquad \| \omega \|_{X^{n-2,\omega}} \lesssim \| \psi \|_{Y^n}
\eeq
and
\beq \label{u4}
\| \phi(y)^{-1} u_2 \|_{X^{n-1,0}} \lesssim \nu^{1/4} |\alpha | \,  \| \psi \|_{Y^n},
\eeq
where $\phi(y) = y / (1 + y)$.
Moreover, for any stream functions $\psi^1$ and $\psi^2$, we have
\beq \label{u5}
\| {\cal Q} (\psi^1,\psi^2) \|_{X^{n-1}_\omega} \lesssim \nu^{1/4} | \alpha | \, \| \psi^1 \|_{Y^n} \| \psi^2 \|_{Y^n} .
\eeq
\end{lemma}
Note the gain of a factor $\nu^{1/4}$ in (\ref{u5}).

\begin{proof} 
The bounds on $u_1$ and $u_2$ directly come from the definition of $Y^n$. 
Moreover, using the divergence free condition, we have $\partial_y u_2 = - \partial_x u_1$,
which leads to (\ref{u4}).
Next, (\ref{u5}) is then a consequence of (\ref{u1}) and (\ref{u4}).
\end{proof}

\begin{definition}
For any $\lambda$, we define ${\cal S}(\lambda)$ to be the set of functions of the form
$$
\psi(t,y) = \sum_{k = 0}^{+\infty} (\sqrt{\nu} t)^ k \psi_k(y) e^{  \nu^{1/2} \lambda t} ,
$$
such that, for any  $n \ge 0$,
$$
\| \psi \|_{n} = \sum_{k=0}^{+\infty} \| \psi_k \|_{Y^n} < + \infty,
$$
and such that only a finite number of functions $\psi_k$ are non identically zero.
\end{definition}

Note the scale of time through $\sqrt{\nu} t$.


\subsection{Inverse of Orr Sommerfeld operator \label{defiGreen}}


The aim of this section is to get bounds on the solution $\psi$ to
\beq \label{Orrff}
Orr_{\alpha,c,\nu}(\psi) = f
\eeq
where $f(y)$ is some source term, and on the solution $\omega$ to 
\beq \label{LNSS}
\partial_t \omega + L_{NS}^\omega \omega =   \Delta_\alpha g,
\eeq
where $g(t,y)$ is  of the form
\beq \label{serieg}
g(t,y) = \sum_{k \ge 0} (\sqrt{\nu} t)^k g_k(y) e^{- i \alpha c t}
\eeq
for some smooth functions $g_k(y)$.

The key ingredient to solve (\ref{Orrff}) is the accurate description of the Green function of Orr Sommerfeld equation, 
obtained in \cite{Bian4,GN3}, and recalled in section \ref{descriptionGreen},
which allows a detailed  study of the inverse of the Orr Sommerfeld operator, away from its eigenvalue $c(\alpha,\nu)$,
close to its eigenvalue $c(\alpha,\nu)$ and at its eigenvalue $c(\alpha,\nu)$. 

Equation (\ref{serieg}) can then be solved using (\ref{Orrff}).
We recall that a factor $i \alpha^{-1} \nu^{-1/4}$ appears when we go from $L_{NS}^\omega$ to $Orr_{c,\alpha,\nu}$.
Note the interplay between vorticity ($f$ and $\omega$) and stream functions ($\psi$ and $g$).

Let us begin by an unformal study of a caricature of (\ref{LNSS}), namely of the ordinary differential equation
\beq \label{caricature}
\partial_t \phi - \eps \phi = e^{\lambda t} \,
\eeq
where $\eps > 0$ is small. If $| \lambda - \eps | \gtrsim \eps$, then a solution of (\ref{caricature}) is
$$
\phi(t) = {e^{\lambda t} \over \lambda - \eps} 
$$
and is of order $O(\eps^{-1})$. If on the contrary $| \lambda - \eps | \lesssim \eps$, a better solution is given by
$$
\phi(t) = {e^{\lambda t} - e^{\eps t} \over \lambda - \eps} = \sum_{n \ge 1} (\lambda - \eps)^{n-1} {t^n \over n!}  e^{\eps t} 
\lesssim \eps^{-1} \sum_{n \ge 1} {(\eps t)^n \over n !} e^{\eps t} .
$$
If $\lambda = \eps$, then one solution is $\phi(t) = \eps^{-1} (\eps t) e^{\eps t}$.
Hence the solution of (\ref{caricature}) is always of size $O(\eps^{-1})$, and the singularity at $\lambda = \eps$
may be handled by introducing polynomials in $\eps t$.


\subsubsection{Description of the Green function \label{descriptionGreen}}


As recalled above, there exist four independent solutions to the Orr Sommerfeld equation (\ref{Orrzero}).
Using these four particular solutions $\phi_{s,\pm}$ and $\phi_{f,\pm}$, the Green function $G_{\alpha,c,\nu}$ of the
Orr Sommerfeld equation, together with its boundary conditions, may be explicitly constructed. 
Namely, following \cite{Bian4,GN3}, we  can decompose $G_{\alpha,c,\nu}$ into
$$
G_{\alpha,c,\nu} = G^{int} + G^b,
$$
where $G^{int}$ is an "interior" Green function, which does not take into account the boundary conditions at $0$, 
and where $G^b$ is a "boundary" Green function, which recovers these conditions.
Moreover, $G^{int}$ is explicitly given by
\beq \label{int}
G^{int}(x,y) = a_+(x)  \phi_{s,+}(y) 
+ {b_+(x) }  {\phi_{f,+}(y) \over \phi_{f,+}(x)} 
\quad \hbox{for} \quad y < x,
\eeq
and by
\beq \label{int2}
G^{int}(x,y) = a_-(x)  \phi_{s,-}(y)
+ {b_-(x)} {\phi_{f,-}(y) \over \phi_{f,-}(x)} 
\quad \hbox{for} \quad y  > x,
\eeq
where, as proved in \cite{Bian4} (taking into account the change of the definition of $\alpha$ by a factor $\nu^{1/4}$),
\beq \label{boundsG1}
a_\pm(x) = O \Bigl( { \alpha \nu^{1/4} \over \nu \mu^2(x)} \Bigr), \qquad b_\pm(x) = O(1).
\eeq
We note that, whereas $b_\pm(x)$ is bounded uniformly in $x$, $a_\pm$ is of order 
$O(\alpha \nu^{-3/4} \gamma^{-2}) = O(\nu^{-1/4})$ if $x$ is of order $\nu^{1/4}$ and
of order $O(1)$ if $\nu^{-1/4} x$ is large. In other word, the response of Orr Sommerfeld equation is large
if the force is close to the critical layer.
Moreover,
$$
\int_0^1 | a_\pm(x) | \, dx < + \infty.
$$
By construction,
$$
Orr_{\alpha,c,\nu} \, G^{int}(x,y) = \delta_x .
$$
As a consequence, for any function $f$, if we define
\beq \label{Green}
\psi^{int}(y) = \int_{\mathbb{R}^+} G^{int}(x,y) f(x) \, dx ,
\eeq
then
\beq \label{Green2}
Orr_{\alpha,c,\nu}(\psi^{int}) = f .
\eeq
In general, $\psi_{int}$ does not satisfy the boundary conditions at $y = 0$, but goes to $0$ at infinity.

To recover the boundary conditions $G_{\alpha,c,\nu}(x,0) = \partial_y G_{\alpha,c,\nu}(x,0) = 0$ on the Green function, we add a linear combination
of $\phi_{s,-}(y)$ and $\phi_{f,-}(y)$, of the form
$$
G^b(x,y) = d_s(x) \phi_{s,-}(y)  + d_f(x) \phi_{f,-}(y) .
$$
The coefficients $d_s(x)$ and $d_f(x)$ must then satisfy
\beq \label{defiGb2}
\left( \begin{array}{cc}
\phi_{s,-}(0) & \phi_{f,-}(0) \cr
\partial_y \phi_{s,-}(0) & \partial_y \phi_{f,-}(0) \end{array} \right)
\left( \begin{array}{c}
d_s(x) \cr d_f(x) \cr \end{array} \right) 
= - \left( \begin{array}{c} 
G^{int}(x,0) \cr \partial_y G^{int}(x,0) \cr \end{array} \right) .
\eeq
Let $M$ be the matrix of the left hand side of (\ref{defiGb2}). Its determinant is
$$
\det M = \phi_{s,-}(0) \partial_y \phi_{f,-}(0) - \partial_y \phi_{s,-}(0) \phi_{f,-} (0)  
$$
and vanishes exactly at the eigenvalue $c(\alpha,\nu)$.
Direct computations \cite{Bian4,GN3} then show that
\beq \label{boundsG2}
d_s(x) = O  \Bigl( { \alpha \nu^{1/4} \over \nu \mu^2(x) } | c - c(\alpha,\nu) |^{-1}  \Bigr), \qquad
d_f(x) = O \Bigl(   | c - c(\alpha,\nu) |^{-1} \Bigr).
\eeq
In particular, $G^{b}$ is singular at $c(\alpha,\nu)$, which is not the case of $G^{int}$ which is bounded everywhere.
Note that $d_s(x)$ is of order $O( | c - c(\alpha,\nu) |^{-1})$ if $\nu^{-1/4} x$ is large, 
and else is of order $O(\nu^{-1/4} | c - c(\alpha,\nu) |^{-1})$. In particular,
\beq \label{boundintds}
\int_0^1 | d_s(x) | \, dx \lesssim | c - c(\alpha,\nu) |^{-1} .
\eeq
It is  easier to understand the geometry of the problem directly on $\psi^{int}$.
We want to add a corrector $\psi^b(y)$ of the form
\beq \label{defipsib}
\psi^b(y) =  d_s \phi_{s,-}(y)  + d_f \phi_{f,-}(y) ,
\eeq
such that
\beq \label{defiGb4}
M
\left( \begin{array}{c}
d_s \cr d_f \cr \end{array} \right) 
= - \left( \begin{array}{c} 
\psi^{int}(0) \cr \partial_y \psi^{int}(0) \end{array}  \right) .
\eeq
When $c = c(\alpha,\nu)$, $M$ is of rank $1$ and the range of $M$ is spanned by $e_1 = (\phi_{s,-}(0),\partial_y \phi_{s,-}(0))$.
Thus, (\ref{defiGb4}) can be solved if and only if $(\psi^{int}(0), \partial_y \psi^{int}(0))$ is colinear to $e_1$.
The range of $Orr_{\alpha,c(\alpha,\nu),\nu}$ is thus characterized by 
\beq \label{charrange}
\phi_{s,-}(0) \partial_y \psi^{int}(0) - \partial_y \phi_{s,-}(0) \psi^{int}(0) = 0,
\eeq
where $\psi^{int}(0)$ and $\partial_y \psi^{int}(0)$ may be seen as linear forms on $f$.
Thus (\ref{charrange}) may be seen as the necessary and sufficient condition for $f$ to be in the range of $Orr_{\alpha,c(\alpha,\nu),\nu}$
and characterizes the range of this operator, of codimension $1$. Let $\psi_3$ be orthogonal to this range.
Then $\psi_3$ is an eigenvector of $Orr^t_{\alpha,c(\alpha,\nu),\nu}$, with corresponding eigenvalue $c(\alpha,\nu)$.

Let $\lambda(\alpha,\nu) = - i \alpha c(\alpha,\nu)$ be an eigenvalue of $L_{NS}$.
The next step is to decompose an arbitrary divergence free vector field $v$ into
$$
v = B_1 v + B_2 v,
$$
where $B_1 v$ is in the kernel of  $L_{NS} + \lambda(\alpha,\nu)$ 
and $B_2 v$ is in the range of $L_{NS} + \lambda(\alpha,\nu)$.

\begin{lemma} \label{propdecomposition}
There exist two linear operators $B_1$ and $B_2$ such that, for any $v$, 
$$
B_1 v \in \ker \, ( L_{NS} + \lambda(\alpha,\nu)),\qquad
B_2 v \in \hbox{\rm Im} \, ( L_{NS} + \lambda(\alpha,\nu)) .
$$
 Moreover, $B_1$ and $B_2$ are bounded operator from  $X^{n,0}$ to $X^{n,0}$ for any $n$.

There  also exists a linear operator $B_3$, inverse of $L_{NS} +  \lambda(\alpha,\nu)$ 
on $\hbox{\rm Im} \,  (L_{NS} + \lambda(\alpha,\nu))$, such that,
for any $v$ in $\hbox{\rm Im } (L_{NS} + \lambda(\alpha,\nu))$,
\beq \label{pseudoinv}
(L_{NS} + \lambda(\alpha,\nu))  \, B_3 v = v .
\eeq
Moreover, $B_3$ is a continuous operator in the sense that, for any integer $n$,
$$
\| B_3 f \|_{X^{n,0}} \lesssim \nu^{-1/4} \| f \|_{X^{n,0}}.
$$
\end{lemma}

\begin{proof}
We already know that $c(\alpha,\nu)$ is a simple eigenvalue, thus the dimension of the kernel 
of $L_{NS} + \lambda(\alpha,\nu)$ is $1$. Let $v_{NS}$ be an eigenvector of
$L_{NS} + \lambda(\alpha,\nu)$. Then $B_1 v$ can be searched under the form $\phi(v) v_{NS}$ for some scalar $\phi(v)$.
Similarly the adjoint operator $L_{NS}^t + \lambda(\alpha,\nu)$ has a kernel of dimension $1$, spanned by some eigenvector $v_{NS}^t$.
As $v_{NS}^t$ is orthogonal to the range of $L_{NS} + \lambda(\alpha,\nu)$, we have
$$
( B_2 v,  v_{NS}^t)  =  0,
$$
thus
$$
\phi(v) = {(v , v_{NS}^t) \over (v_{NS}, v_{NS}^t)}.
$$
In particular,
\beq \label{defiB1}
B_1 v =  {(v , v_{NS}^t) \over (v_{NS}, v_{NS}^t)} v_{NS}, \qquad B_2 v = v - B_1 v.
\eeq
We note, using (\ref{ortho}, that $(v_{NS},v_{NS}^t)$ does not vanish \cite{Bian1} and is of order $O(1)$,  thus,
$$
| \phi(v) | \lesssim \| v \|_{X^{n,0}} .
$$
In particular, $B_1$ and $B_2$ are bounded operators from $X^{n,0}$ to $X^{n,0}$ for any $n \ge 0$.

We now construct a pseudo inverse of $L_{NS} + \lambda$, defined on its image. 
Let $v$ be in the range of $L_{NS} + \lambda$. Let $\omega$ be the vorticity of $v$. We have to solve
$$
Orr_{\alpha,c,\nu} \psi = i {\omega \over \alpha \nu^{1/4}} .
$$
Then, by definition, $\omega$ satisfies
 (\ref{charrange}). As a consequence we can solve the equation (\ref{defiGb4}), and the scalars
 $d_s$ and $d_f$ are bounded by $\nu^{-1/4}$. Thus $\psi_b$ defined by (\ref{defipsib}) is also bounded by $\nu^{-1/4}$.
We then define $B_3 f$ by
$$
B_3 f = \nabla^\perp( \psi^{int} + \psi^b),
$$
where $\psi^{int}$ is defined by (\ref{Green}).
By construction, (\ref{pseudoinv}) is satisfied, which ends the proof. 
\end{proof}


\subsubsection{Inverse when $| c - c(\alpha,\nu) | \gtrsim \nu^{1/4}$}


In all this paragraph we assume that $| c - c(\alpha,\nu) | \gtrsim \nu^{1/4}$.
The solution of (\ref{Orrff}) is then  given by 
\beq \label{Green10}
\psi(y) = \int_{\mathbb{R}^+} G_{\alpha,c,\nu}(x,y) f(x) \, dx .
\eeq
The explicit expression of $G_{\alpha,c,\nu}$, together with (\ref{boundsG1}) and (\ref{boundsG2}), allows us to bound $\omega$ and 
$\psi$ in terms of $f$.

\begin{lemma} \label{away}
For any integer $n$, we have, for any $f(y) \in X^n_\omega$,
\beq \label{away1} 
\| Orr^{-1}_{\alpha,c,\nu} f \|_{Y^n} \lesssim   \nu^{-1/4}  \| f \|_{X^n_\omega} 
\eeq
provided $| c - c(\alpha,\nu) | \gtrsim \nu^{1/4}$.
\end{lemma}

\begin{proof}
We  split $G_{\alpha,c,\nu}$ into $G^{int}$ and $G^b$. The Green function $G^{int}$ leads to the contribution
\begin{equation*}
\begin{split}
\psi^{int}(y) &=  \int_{x \le y} a_-(x) f(x) \phi_{s,-}(y) \, dx 
+   \int_{x \ge y} a_+ (x) f(x) \phi_{s,+}(y) \, dx\\
&\quad+   \int_{x \le y} b_-(x) f(x) {\phi_{f,-}(y) \over \phi_{f,-}(x)} \, dx 
+   \int_{x \ge y} b_+ (x) f(x) {\phi_{f,+}(y) \over \phi_{f,+}(x)} \, dx.
\end{split}
\end{equation*}
The first term of $\psi^{int}(y)$ may be bounded by
$$
  | \phi_{s,-}(y) | \, \| f \|_{X^n_\omega} \int_{\rit^+} | a_-(x) | e^{-C_0 \nu^{1/4} x} \, dx \lesssim   \| f \|_{X^n_\omega} e^{- C_0 \nu^{1/4} |y |} 
 $$
since the integral of $a_-(x) e^{-C_0 x}$ is bounded. 
The second term is bounded by
$$
| \phi_{s,+}(y) | \, \int_{x \ge y} |a_+(x)| \, | f(x) | \, dx .
$$
If $ y \le 1$, as the integral of $a_+(x)$ is bounded, this integral is bounded by 
$$
 e^{C_0 \nu^{1/4} y} \| f \|_{X^n_\omega} \int_{\rit^+} | a_+(x) | e^{-C_0 x} \, dx  \lesssim \| f \|_{X^n_\omega} 
  \lesssim   \| f \|_{X^n_\omega}  e^{- C_0 \nu^{1/4} y}.
$$
If on the contrary $y \ge 1$, then $a_+$ is bounded, leading to a bound by
\begin{equation*} \begin{split}
 | \phi_{s,+}(y) | \, \int_{x \ge y}  | f(x) | \, dx 
& \lesssim   | \phi_{s,+}(y) |  \,  \| f \|_{X^n_\omega}
\int_{x \ge y} e^{- C_0 x} \, dx \lesssim e^{C_0 \nu^{1/4} y - C_0 y}   \| f \|_{X^n_\omega}
\\&\lesssim e^{- C_0 \nu^{1/4} y}  \| f \|_{X^n_\omega}
\end{split} \end{equation*}
since  $|f(x)| \le \| f \|_{X^n_\omega} e^{-C_0 x}$.

Let us turn to the third integral defining $\psi^{int}(y)$. Using (\ref{decrease}), 
this third integral is bounded by
$$
 \| f \|_{X^n_\omega}  \int_{x \le y} e^{- C_0 x - C_1 (y - x)} \, dx 
 \lesssim  \| f \|_{X^n_\omega}  e^{- C_0 y} 
 \lesssim   \| f \|_{X^n_\omega}   e^{- C_0 \nu^{1/4} y}.
 $$
The fourth integral is similar.

The term in $G^b$ leads to a contribution 
$$
\psi^b(y) =\phi_{s,-}(y)  \int d_s(x) f(x) \, dx 
+\phi_{f,-}(y)\int d_f(x) f(x) \, dx .
$$
As $d_f(x)$ is bounded we have
$$
| \psi^b(y) | \le { | \phi_{s,-}(y)| \, \| f \|_{X^n_\omega}   \over | c - c(\alpha,\nu) |} \int_{\rit_+} | d_s(x) | e^{-C_0 x} \, dx 
+  { | \phi_{f,-}(y) | \, \| f \|_{X^n_\omega}  \over | c - c(\alpha,\nu) |}  \int_{\rit_+} e^{-C_0 x} \, dx.
$$
As the integral of $d_s(x) e^{-C_0 x}$ is finite, we obtain
$$
| I_2 | \lesssim   { e^{- C_0 \nu^{1/4} y} \over | c - c(\alpha,\nu) |}  \| f \|_{X^n_\omega}
\lesssim   \nu^{-1/4}  e^{-  C_0 \nu^{1/4}  y}  \| f \|_{X^n_\omega} ,
$$
which ends the pointwise bound on $Orr^{-1}_{\alpha,c,\nu} f(y)$.
Derivatives may be bounded in a similar way.
Note that the $\nu^{-1/4}$ factor comes from the "boundary layer part" $G^b$ of the Green function only.
\end{proof}

We now turn to the resolution of (\ref{LNSS})

\begin{lemma}
If $g(t,y) \in S(\lambda)$, then the solution $\psi_\alpha$ to (\ref{LNSS})  lies in ${\cal S}(\lambda)$,
and for any positive integer $n$,
$$
\| \psi_\alpha \|_{n} \lesssim \nu^{-1/2} \|  g \|_{n}  .
$$
\end{lemma}

\begin{proof}
It is sufficient to prove this Lemma when $g$ is of the form
$$
g(t,y) = (\sqrt{\nu} t)^Q e^{\nu^{1/2}  \lambda t} g_Q(y).
$$
We look for $\psi_\alpha$ under the form
$$
\psi_\alpha(t,y) = \sum_{k \le Q} (\sqrt{\nu} t)^k   e^{\nu^{1/2} \lambda t} \psi_{\alpha,k}(y) ,
$$
which leads to
$$
\psi_{\alpha,Q} = - Orr_{\alpha,c,\nu}^{-1} \Bigl[i  \alpha^{-1} \nu^{-1/4} \Delta_\alpha  g_Q \Bigr],
$$
which is of norm of order $\nu^{-1/2}$, 
and, for $k < Q$, to
$$
\psi_{\alpha,k} = - Orr_{\alpha,c,\nu}^{-1}  \Bigl[
\sqrt{\nu} (k+1) i \alpha^{-1} \nu^{-1/4} \Delta_\alpha \psi_{\alpha,k+1} \Bigr],
$$ 
which are also of norm of order $\nu^{-1/2}$.
The Lemma follows by induction. 
\end{proof}


\subsubsection{Inverse for $c$ close to $c(\alpha,\nu)$}


We now describe the geometry of $Orr_{\alpha,c,\nu}$ when $c$ is close to $c(\alpha,\nu)$.

\begin{lemma} \label{propdecomposition2}
If $g  \in {\cal S}(\lambda)$
 then there exists a solution $\psi$ to (\ref{LNSS}) of the form $\psi = \psi_1 + \psi_2$ where
$\psi_1 \in {\cal S}(\lambda)$ and $\psi_2 \in {\cal S}(\lambda(\alpha,\nu))$.
Moreover, for any positive integer $n$,
$$
\| \psi_1 \|_n + \| \psi_2 \|_n \lesssim  \nu^{-1/2} \| g \|_{n}   . 
$$
\end{lemma}

\begin{proof}
The first step is to construct an inverse of $L_{NS} - \lambda$ for $\lambda$ close to $\lambda(\alpha,\nu)$.
Let $v = \nabla^\perp g$ be the velocity associated to the stream function $g$.
We decompose $v$ in $v = B_1 v + B_2 v$ and define an approximate solution ${\cal S}^{app}$
by 
$$
{\cal S}^{app} v = {B_1 v \over \lambda - \lambda(\alpha,\nu)} + B_3 B_2 v 
$$
in such a way that
$$
(L_{NS} + \lambda) \, {\cal S}^{app} v = v + [ \lambda - \lambda(\alpha,\nu) ] B_3 B_2 v .
$$
We then construct an exact inverse ${\cal S}_{\alpha,c,\nu}$ of $Orr_{\alpha,c,\nu}$ using the iterative scheme
$$
v_{n+1} =  - [ \lambda - \lambda(\alpha,\nu) ]  \, B_3 \, B_2 \, v_n
$$
starting from $v_0 = v$.
We then set
$$
{\cal S}_{\alpha,c,\nu} v =  {  {\cal S}_{\alpha,c,\nu} ^1 v \over \lambda - \lambda(\alpha,\nu)} 
+   {\cal S}_{\alpha,c,\nu}^2 v
$$
where
$$
{\cal S}^1_{\alpha,c,\nu} v = \sum_{n \ge 0}   B_1 v_n 
$$
and
$$
{\cal S}^2_{\alpha,c,\nu} v = \sum_{n \ge 0}  B_3 B_2 v_n.
$$
Note that these series converge since $B_3$ is bounded by $\nu^{-1/4}$ and $| \lambda - \lambda(\alpha,\nu)|$
is of order at most $\nu^{1/2}$.
We observe that
$$
(L_{NS} + \lambda) \,  {\cal S}_{\alpha,c,\nu} v =   v .
$$
Moreover, ${\cal S}^1_{\alpha,c,\nu}$ is a bounded operator, whereas the norm of ${\cal S}^2_{\alpha,c,\nu}$ is of order
$\nu^{-1/4}$.
Let us now solve 
$$
\partial_t w + L_{NS} w = 
(\sqrt{\nu} t)^N e^{-i \nu^{1/4} \alpha c t}  v.
$$
We look for solutions of the form
\beq \label{complex}
w = \sum_{0 \le k \le N} (\sqrt{\nu} t)^k e^{- i \nu^{1/4} \alpha c t} w_k^1
+  \sum_{0 \le k \le N} (\sqrt{\nu} t)^k e^{- i \nu^{1/4} \alpha c(\alpha,\nu) t} w_k^2.
\eeq
For $k = N$, we choose
\beq \label{omegaQ}
w_N(t,y) = ( {\cal S}^1_{\alpha,c,\nu}  w )
{e^{- i \nu^{1/4}  \alpha c t} -   e^{- i \nu^{1/4} \alpha c(\alpha,\nu) t} \over - i  \alpha \nu^{1/4} (c(\alpha,\nu) - c)} 
+( {\cal S}^2_{\alpha,c,\nu} w)  e^{- i  \nu^{1/2} \alpha c t}.
\eeq
Note that $\alpha \nu^{1/4}  (c - c(\alpha,\nu)) t$ is at most of order $\nu^{1/2} \nu^{\beta} \nu^{-1/2} \log \nu^{-1} \ll 1$.
Thus we may expand 
$$
{ e^{- i \nu^{1/4} \alpha c t} - e^{ - i \nu^{1/4} \alpha c(\alpha,\nu) t} \over  \alpha \nu^{1/4} (c(\alpha,\nu) - c)} 
= e^{ - i \nu^{1/4} \alpha c(\alpha,\nu) t} {e^{- i \nu^{1/4} \alpha (c - c(\alpha,\nu)) t} - 1 \over \alpha \nu^{1/4} ( c(\alpha,\nu) - c)} 
$$
$$
= - i {(\sqrt{\nu} t)  \over \nu^{1/2}} e^{ - i \nu^{1/4} \alpha c(\alpha,\nu) t}
+ {1 \over \sqrt{\nu}}\sum_{n \ge 1} {(-i)^n \over n!}  \alpha^{n-1} \nu^{(n-1)/4} (c - c(\alpha,\nu))^{n-1}
 \nu^{1/2} t^n e^{ - i \nu^{1/4} \alpha c(\alpha,\nu) t}.
$$
The series is convergent and can be put in the second term of the right hand side of (\ref{complex}).
The second term of the right hand side of (\ref{omegaQ}) can be put in the first term of the right hand side of (\ref{complex}).
We then define $w_k^1$ and $w_k^2$ for $k < N$ by induction.
\end{proof}


\subsection{Inverse for $c = c(\alpha,\nu)$}


\begin{lemma} \label{propdecomposition2-2}
If $g \in {\cal S}(\lambda(\alpha,\nu))$ 
 then there exists a solution $\psi$ to (\ref{LNSS}) 
 such that $\psi \in {\cal S}(\lambda(\alpha,\nu))$.
Moreover, for any positive integer $n$,
$$
\| \psi \|_{n}  \lesssim \nu^{-1/2} \| g \|_{n} .
$$
\end{lemma}

\begin{proof}
For $c = c(\alpha,\nu)$, by definition, $Orr_{\alpha,c(\alpha,\nu),\nu}$ is not invertible. If we want to solve
$$
\partial_t v + L_{NS} \, v = e^{i \alpha (x - c(\alpha,\nu) t) } w(y),
$$
we first split $w$ into $B_1 w + B_2 w$, and observe that the solution is explicitly given by
$$
v = \nu^{-1/2}  (\sqrt{\nu}t) e^{i \alpha (x - c(\alpha,\nu) t)}   B_1 w
+   e^{i \alpha (x - c(\alpha,\nu) t)}  B_3 B_2 w  .
$$
More generally, if $N$ is an integer, then the solution to
\beq \label{poly}
\partial_t v + L_{NS} \, v =   (\sqrt{\nu}t)^N  e^{i \alpha (x - c(\alpha,\nu) t) } w(y)
\eeq
is of the form
$$
v = \sum_{0 \le k \le N+1} (\sqrt{\nu} t)^k e^{i \alpha (x - c(\alpha,\nu) t) } w_k(y),
$$
with 
$$
w_{N+1} = \nu^{-1/2} {B_1 w \over N+1},
$$
the other $w_k$ being obtained by induction.
The end of the proof is then straightforward, using the bounds on $B_1$, $B_2$ and $B_3$.
\end{proof}


\section{"Slow" instabilities  \label{general}}


The aim of this section is to construct instabilities of large scale flows. More precisely,
let $u^\eps(t,x,y)$ be a sequence of solutions of Navier Stokes equations which slowly depend on time,
 have "large structures" in the $x$ and $y$ variables, and a "boundary layer behavior" near $y = 0$,
 namely let $u^\eps(t,x,y)$ be a sequence of solutions of the form
\beq \label{ueps}
u^\eps(t,x,y) = 
\left( \begin{array}{c} U_1^{int,\eps}(\eps t,\eps x,\eps y) + U_1^{bl,\eps}(\eps t, \eps x, y) \cr
U_2^{int,\eps}(\eps t,\eps x,\eps y) + \eps U_2^{bl,\eps}(\eps t, \eps x, y)
\end{array} \right)
\eeq
where $\eps$, which depends on $\nu$, goes to $0$ as $\nu \to 0$. 
We assume that $U_1^{int,\eps}$ and $U_2^{int,\eps}$ are smooth in their three variables
$T = \eps t$, $X = \eps x$ and $Y = \eps y$
and converge to some functions $U_1^{int,0}$ and $U_2^{int,0}$ in $C^\infty$ as $\eps \to 0$.

 We moreover assume that $U_1^{bl,\eps}$ and $U_2^{bl,\eps}$ are smooth in their three variables $T$, $X$ and $y$, 
 are exponentially decaying in $y$ (as well as all their derivatives),  and
 converge in $C^\infty$ (with uniform exponential decay) to some functions $U_1^{b,0}$ and $U_2^{b,0}$.

We define
$$
U_1^\eps(t,x,y) = U_1^{int,\eps}(\eps t,\eps x,\eps y) + U_1^{bl,\eps}(\eps t, \eps x, y)
$$
and
$$
U_2^\eps(t,x,y) = U_2^{int,\eps}(\eps t,\eps x,\eps y) + \eps U_2^{bl,\eps}(\eps t, \eps x, y),
$$
and similarly for $U_1^0$ and $U_2^0$.

This  study may be extended to the case where $u^\eps(t,x,y)$ depends on multiple scales $\eps_1 \ll ... \ll \eps_n \ll 1$,
where $\eps_n \to 0$ as $\nu \to 0$.
Note the $\eps$ factor in front of $U_2^{bl,\eps}$, which comes from the incompressibility condition.
Moreover, Prandtl's Ansatz is of the form (\ref{ueps}), up to a rescaling of $t$, $x$ and $y$ by a factor $\nu^{1/2}$.

\medskip

We will distinguish two kinds of instabilities: "slow" instabilities, which occur over times of order $O(\nu^{-1/2})$ (corresponding
to the second kind of instabilities discussed in the introduction), and "fast" instabilities, which occur over times of order $O(1)$ 
(corresponding to the first kind of instabilities discussed in the introduction).

We will prove that, if $(U_1^{bl,0}(0,0,0,y),0)$ is an unstable shear layer, then $u^\eps$ is also unstable, 
provided $\eps(\nu) \lesssim \nu$ in the first
case and $\eps(\nu) \ll \nu^\delta$ for some arbitrarily small and positive $\delta$ in the second one.

In this section we focus on "slow" instabilities and prove the following general instability result.
We recall that $\alpha$ has been rescaled by a factor $\nu^{1/4}$.

\begin{theo} \label{theogeneral}
Let us assume that $\eps(\nu) \lesssim \nu^\delta$ for some $\delta > 3/4$.
Let us assume that 
$$
\partial_y U_1^{bl,0}(0,0,0,0) \ne 0.
$$ 
Let $\theta > 0$ be arbitrarily small.
Then,  
for any arbitrarily large $N$ and any arbitrarily large $s$, and for $\nu$ small enough,
there exists a solution $v^\nu$ of Navier Stokes equations with forcing term $f^\nu$,
 a time $T^\nu$ and a constant $\sigma > 0$, such that
 \beq \label{theogeneral1-3}
\| v^\nu(0,\cdot,\cdot) - u^{\eps(\nu)}(0,\cdot,\cdot) \|_{H^s} \le \nu^N,
\eeq
\beq \label{theogeneral2-3}
\| f^\nu \|_{L^\infty([0,T^\nu],H^s)} \le \nu^N
\eeq
and
\beq \label{theogeneral3-3}
\| v^\nu(T^\nu,\cdot,\cdot) - u^{\eps(\nu)}(T^\nu,\cdot,\cdot) \|_{L^\infty} \ge \sigma \nu^{1/4 + \theta} ,
\eeq
where 
$$
T^\nu \sim C_0 \nu^{-1/2} \log \nu^{-1}
$$
for some constant $C_0$.
\end{theo}

\subsubsection*{Remark}
If $u(t,x,y)$ is independent in $x$, we can look for periodic instabilities. In this case there is no need to "localise" the instability
by creating a "wave packet", and it is sufficient to work with plane waves of the form $e^{i \nu^{1/4} \alpha + \nu^{1/2} \lambda t}$.
The proof is then  simpler.


\subsection{Rewriting  Navier Stokes equations}


The first step is to rewrite the linearised Navier Stokes equations near $u$.
Let 
$$
T = \eps t, \qquad X = \eps x, \qquad Y = \eps y,
$$
and let $\Omega^\eps(T,X,Y,y)$ be the vorticity of $u^\eps$, namely
$$
\Omega^\eps(T,X,Y,y) = \Omega_1^\eps(T,X,Y,y) + \eps \Omega_2^\eps(T,X,Y,y),
$$
where
$$
\Omega_1^\eps(T,X,Y,y) = - \partial_y U_1^{bl,\eps}(T, X, y)
$$
and
$$
\Omega_2^\eps(T,X,Y,y) = - \partial_Y U_1^{int,\eps}(T,X,Y) +  \partial_X U_2^{int,\eps} (T,X,Y) 
+ \eps \partial_X U_2^{bl,\eps}(T,X,y).
$$
Then 
$$
L_{NS}^\omega [ u^\eps(t) ] \omega = - U_1^\eps \partial_x  \omega - \eps U_2^\eps \partial_y \omega  
- \eps v_1 \partial_X \Omega^\eps - v_2 \partial_y \Omega^\eps - \eps v_2 \partial_Y \Omega^\eps
+ \nu \Delta \omega , 
$$
thus
\beq \label{decompLNS}
L_{NS}^\omega [u(t) ] \, \omega = L_{NS}^0 \, \omega + L_{NS}^1  [ u^\eps(t) ] \, \omega
\eeq
where
$$
L_{NS}^0 \,  \omega = - U_1^\eps(0,0,0,y) \partial_x  \omega - v_2 \partial_y \Omega_1^\eps(0,0,0,y) + \nu \Delta \omega ,
$$
and
\begin{equation*}
\begin{split}
L_{NS}^1  [ u^\eps(t) ] \,  v =& - \Bigl[ U_1^\eps(T,X,Y,y) - U_1^\eps(0,0,0,y) \Bigr] \partial_x \omega
- v_2 \partial_y \Bigl[ \Omega_1^\eps(T,X,Y,y) - \Omega_1^\eps(0,0,0,y) \Bigr] \\&
-  \eps U_2^\eps \partial_y \omega 
- \eps v_1 \partial_X \Omega^\eps 
- \eps v_2 \partial_y \Omega_2^\eps
- \eps v_2 \partial_Y \Omega^\eps .
\end{split}
\end{equation*}
Note that $L_{NS}^0$ is simply the linearized Navier Stokes operator near the "frozen" profile  
$$
U_0^\eps = (U_1^\eps(0,0,0,y),0),
$$
 and that $L_{NS}^1 [ u^\eps(t) ]$ takes into account the time and space variations of $u^\eps(t,x,y)$ 
 over times and scales of size $\eps^{-1}$, together with  contributions order $O(\eps)$.

This decomposition of $L_{NS}^\omega[u^\eps(t)]$ leads to the following rewriting of Navier Stokes equations
\beq \label{rewriteNS}
\partial_t \omega + L_{NS}^0 \omega = - L_{NS}^1 [ u^\eps(t) ] \, \omega  - Q(\psi,\psi).
\eeq
We now truncate (\ref{rewriteNS}) both in time and space.
Let $\chi(u)$ be a smooth function which equals $1$ for $| u | \le 1$ and $0$ for $| u |  \ge 2$. Let $\gamma >  \beta + 1/4$.
We can choose $\beta$ arbitrarily close to $1/4$ and thus $\gamma$ arbitrarily close to $1/4$.
We truncate (\ref{rewriteNS}) in time and space and rewrite it under the form 
\beq \label{rewriteNS1}
\partial_t \omega + L_{NS}^0 \omega = - \widetilde L_{NS}^1 \omega - Q(\psi,\psi) + {\cal E} 
\eeq
where
$$
\widetilde L_{NS}^1 \omega =  \chi (\nu^{2\gamma} t) \chi( \nu^\gamma x) \chi(\nu^\gamma y)  L_{NS}^1 [ u(t)  ] \omega
$$
and
\beq \label{rewriteNS2}
{\cal E} = \chi(\nu^{2\gamma} t) \Bigl( 1 -\chi( \nu^{\gamma} x) \chi(\nu^\gamma y) \Bigr) L_{NS}^1 [u(t)] \omega .
\eeq
Note that we truncate $x$ at $\nu^{-\gamma}$, which is slightly larger than $\nu^{-1/4}$, whereas we truncate time 
at $\nu^{- 2 \gamma}$, which is slightly larger than $\nu^{-1/2}$.
We note that the solutions of (\ref{rewriteNS}) and (\ref{rewriteNS1}) are equal when $t \le \nu^{-2 \gamma}$.
The term ${\cal E}$ will later be  put into the forcing term $f^\nu$.
We first bound $\widetilde L_{NS}^1$.

\begin{lemma} \label{gain}
Let 
\beq \label{definitionV1}
V_1(t,x,y) = \chi (\nu^{2 \gamma} t) \chi( \nu^\gamma x) \chi(\nu^\gamma y) 
\Bigl[ U_1^\eps(T,X,Y,y) - U_1^\eps(0,0,0,y) \Bigr].
\eeq
Then for any arbitrarily large $N$, $V_1$ may be decomposed in
 $$
 V_1 =  V_1^1 +  V_1^2,
 $$
 where
 the Fourier Laplace transform $\widehat V_1^1$ of $V_1^1$ satisfies
\beq \label{decayhatuchi}
|  \widehat V_1^1(\tau,\alpha,y) | \lesssim  \Bigl( { \eps \over \nu^{2 \gamma }} \Bigr)  {1 \over \nu^{3 \gamma}}
(1 + | \nu^{-2\gamma} \lambda |)^{-N} (1 + | \nu^{-\gamma} \alpha |)^{-N}
\eeq
and where 
$$
|  V_1^2 (t,x,y) | \lesssim \nu^N.
$$
A similar result holds true when $U_1^\eps$ is replaced by $\Omega_1^\eps$ in (\ref{definitionV1}).
\end{lemma}

\begin{proof}
We expand $U_1(T,X,Y,y) - U_1(0,0,0,y)$ in Taylor series and write
$$
V_1(t,x,y) =  V_1^1(t,x,y) +  V_1^2(t,x,y)
$$
where
$$
V_1^1(t,x,y) = \chi (\nu^{2\gamma} t) \chi( \nu^\gamma x) \chi(\nu^\gamma y)  
\sum_{0 \le k,l,m \le P, \, k + l + m > 0} \eps^{k + l + m} A_{k,l,m}^\eps t^k x^l y^m 
$$
for some bounded coefficients $A_{k,l,m}^\eps$, and where
$$
|  V_1^2(t,x,y) | \lesssim  \nu^N
$$ 
provided $P$ is large enough. Terms involving $V_1^2$ will be put in the forcing term.
Now, the Fourier Laplace transform of
$$
\chi (\nu^{2\gamma} t) \chi( \nu^\gamma x) \chi(\nu^\gamma y) \eps^{k+l+m} t^k x^l y^m
$$
is in the form of
$$
\int_{\rit} \int_{\rit} \chi (\nu^{2\gamma} t) \chi( \nu^\gamma x) \chi(\nu^\gamma y) \eps^{k+l+m} t^k x^l y^m
e^{- i \alpha x - \lambda t} \, dx  \, dt,
$$
which is bounded by
$$
\Bigl( { \eps \over \nu^{2 \gamma}} \Bigr)^k 
 \Bigl( { \eps \over \nu^{\gamma}} \Bigr)^l  \Bigl( { \eps \over \nu^{\gamma}} \Bigr)^m
 {1 \over \nu^{3 \gamma}} {1 \over (1 + | \nu^{-2 \gamma} \lambda |)^N (1 + | \nu^{-\gamma} \alpha |)^N} 
$$
for any positive integer $N$.
Thus, the Fourier Laplace $\widehat V_1^1$ of $V_1^1$ satisfies (\ref{decayhatuchi}) for any arbitrarily large $N$,
which ends the proof.
\end{proof}


\subsection{Structure of the instability }


We will construct the approximate solution $V^\nu$ starting from the linear instability of the shear flow
$$
U_0^\eps(y) = (U_1^\eps(0,0,0,y),0).
$$
As $\eps$ goes to $0$, $U_0^\eps(y)$ converges to $(U_1^0(0,0,0,y),0)$. As 
$$
\partial_y U_1^0(0,0,0,0) = \partial_y U_1^{bl,0}(0,0,0,0) \ne 0,
$$
for $\eps$ small enough, $\partial_y U_1^\eps(y) \ne 0$, thus $U_0^\eps(y)$ is an unstable shear layer.
We choose $\alpha_0^\eps$ to be its most unstable eigenmode, namely
$$
\lambda(\alpha_0^\eps,\nu) = \max_\alpha \lambda(\alpha,\nu)
$$
and define $\lambda_r^\eps$ and $\lambda_i^\eps$ by
$$
\lambda(\alpha_0^\eps) = \lambda_r^\eps + i \lambda_i^\eps.
$$
For any $\alpha$, we  moreover define $\lambda_i(\alpha)$ and $\lambda_r(\alpha)$ by
$$
\lambda(\alpha) = \lambda_r(\alpha) + i \lambda_i(\alpha).
$$
We will construct $V^\nu$ under the form
\begin{equation} \begin{split}  
V^\nu(t,x,y) =& \, U_0  + V_{1,1}^\nu(t,x,y) e^{i \alpha_0^\eps \nu^{1/4} x + i \nu^{1/2} \lambda_i^\eps  t + \nu^{1/2} \lambda_r^\eps t}
\\&+  V_{-1,1}^\nu(t,x,y) e^{- i \alpha_0^\eps \nu^{1/4} x - i \nu^{1/2} \lambda_i^\eps t + \nu^{1/2} \lambda_r^\eps t}
\\&+ \sum_{m \in \nit, n \ge 2} V_{m,n}^\nu(t,x,y) 
e^{ i m \alpha_0^\eps \nu^{1/4} x  + i m \nu^{1/2} \lambda_i^\eps t + n \nu^{1/2} \lambda_r^\eps t} ,
\end{split} \end{equation}
where $V^\nu_{-1,1}(t,x,y)$, $V^\nu_{1,1}(t,x,y)$ and all the $V^\nu_{m,n}(t,x,y)$ are series in $\sqrt{\nu} t$, namely of the form
\beq \label{formVnu2}
V_{m,n}^\nu(t,x,y) = \sum_{p \ge 0} (\sqrt{\nu} t)^p V_{m,n,p}^\nu(t,x,y)
\eeq
where, for any integers $m$, $n$ and $p$, $V^\nu_{m,n,p}$ has its Fourier Laplace transform supported in a very small ball around $0$,
of radius of order $\nu^{2 \gamma}$.
Moreover, each $V_{m,n,p}^\nu$ may be expanded in fractional powers of $\nu$.
We only need to construct an approximate solution, thus all the series are in fact be finite sums.
Moreover, as $V^\nu$ is real, we have, for any $m$ and $n$,
\beq \label{formVnu4}
V_{-m,n}^\nu = \bar V_{m,n}^\nu .
\eeq
Let us now describe the form of the various functions which appear in the exrpansion of $V^\nu_{m,n,p}$.
If $m = 1$, $n = 1$, and for any $p$, these functions are of the form
\beq \label{form1} \begin{split}
V(t,x,y) =& \int_{\rit \times \rit^P \times \rit^P}  \phi \Bigl( {\alpha - \alpha_0^\eps \over \nu^\beta}, {\tilde \alpha_1}, 
\cdots, {\tilde \alpha_P }, 
{\tilde \lambda_1 }, \cdots, {\tilde \lambda_P }, y \Bigr)
\\& \times e^{ i \nu^{1/4} (\alpha - \alpha_0^\eps) x + i \nu [\lambda_i(\alpha) - \lambda_i^\eps] t 
+ \nu^{1/2} [\lambda_r(\alpha) - \lambda_r^\eps] t + i  \nu^{\gamma} \sum_{j=1}^P \tilde \alpha_j  x
+ \nu^{2 \gamma} \sum_{j=1}^P \tilde \lambda_j t }
\\& 
\, d\alpha \, d\tilde \alpha_1 \cdots \, d\tilde \alpha_P \, d\tilde \lambda_1 \, \cdots \, d\tilde \lambda_P,
\end{split} \eeq
where the stream function $\phi$ is smooth in all its parameters, compactly supported in $[-1,1]$ in its first parameter,
rapidly decreasing in all its parameters
excepted $y$, and where $P$ is an integer ($P=0$ meaning that there is no variable $\tilde \alpha$ and $\tilde \lambda$
and no integration on these variables).

Their form for $m = -1$ is similar, up to the change $\alpha_0^\eps \to - \alpha_0^\eps$.
We recall that $V_{m,n,p} = 0$ if $m = \pm 1$ and $| n | \ne 1$.

If $| m | > 1$, they are of the form
\beq \label{form2} \begin{split}
V(t,x,y) =&  \int_{\rit^N \times \rit^P \times \rit^P}  \phi \Bigl( {\alpha_1 \pm \alpha_0^\eps \over \nu^{\beta}},\cdots,
{\alpha_N \pm \alpha_0^\eps \over \nu^{\beta}}, {\tilde \alpha_1 }, 
\cdots, {\tilde \alpha_P}, 
{\tilde \lambda_1}, \cdots, {\tilde \lambda_P }, y \Bigr)
\\& \times e^{ i \nu^{1/4} \sum_{j=1}^N (\alpha_j \pm \alpha_0^\eps) x 
+ i \nu^{1/2} \sum_{j=1}^N  \pm [\lambda_i(\alpha_j) - \lambda_i^\eps] t + \nu^{1/2} \sum_{j=1}^N [\lambda_r(\alpha) - \lambda_r^\eps] t}
\\& \times e^{i  \nu^{\gamma} \sum_{j=1}^P \tilde \alpha_j  x
+ \nu^{2 \gamma} \sum_{j=1}^P \tilde \lambda_j t } 
\, d\alpha_1 \, \cdots \, d\alpha_N \, d\tilde \alpha_1 \cdots \, d\tilde \alpha_P \, d\tilde \lambda_1 \, \cdots \, d\tilde \lambda_P,
\end{split} \eeq
where $\phi$ is a smooth function of all its arguments, with compact support in $[-1,1]$ for its $N$ first arguments
and rapidly decaying in all  its arguments excepted $y$,
where there are $N_1$ "+" signs and $N_2$ "-" signs, in such a way that $N_1 + N_2 = n$ and $N_1 - N_2 = m$.

We first give an upper bound on functions of the form (\ref{form1}) and (\ref{form2}).

\begin{lemma}
If $V$ is of the form (\ref{form1}) or (\ref{form2}), 
where $\phi$ is a smooth function of its arguments, with compact support for all its arguments except $y$,
then
\beq \label{boundpsin}
| V(t,x,y) | \lesssim  \langle \sqrt{\nu} t \rangle^{-N/2} 
\eeq
for all $t$, $x$ and $y$,and similarly for all its $x$ and $y$ derivatives.
\end{lemma}

\begin{proof}
To alleviate the notations, we just detail the proof when there are only "$+$" signs.
The decay bound (\ref{boundpsin}) is obtained through a stationary phase argument. The phase is
$$
\Theta(\alpha_1,\cdots,\alpha_N) =  i \nu^{1/4} \sum_{j=1}^N (\alpha_j - \alpha_0^\eps) x 
+ i \nu^{1/2} \sum_{j=1}^N  [\lambda_i(\alpha_j) - \lambda_i^\eps] t + \nu^{1/2} \sum_{j=1}^N [\lambda_r(\alpha_j) - \lambda_r^\eps] t.
$$
It is stationary when all the $\alpha_j$ equal $\alpha_0^\eps$, which achieves the maximum of $\lambda_r(\alpha)$, 
and when $x / t = c_\sigma(\alpha_0^\eps)$, 
the corresponding group phase velocity. In this case $\lambda_r(\alpha_j) - \lambda_r^\eps = 0$ for all $j$.
The variables $\tilde \alpha_j$ and $\tilde \lambda_j$ can be seen as simple parameters.
\end{proof}

We will also need a localization property in the $x$ variable. 

\begin{lemma}
If $V$ is of the form (\ref{form1}) or (\ref{form2}), then, for any $Q$ arbitrarily large,
\beq \label{generalform4}
| V(t,x,y) | \lesssim {1  \over 1 + | \nu^{\beta} x |^Q}.
\eeq
\end{lemma}

\begin{proof}
To alleviate the notations, we just detail the proof when there are only "$+$" signs.
Let us compute the Fourier transform of $V$ in the $x$ variable. We have
\beq  \begin{split}
\widehat V(t,\alpha,y) =&  \int_{\rit^{N-1} \times \rit^P \times \rit^P}  \phi \Bigl( {\alpha_1 - \alpha_0^\eps \over \nu^{\beta}},\cdots,
{\alpha_N - \alpha_0^\eps \over \nu^{\beta}}, {\tilde \alpha_1}, 
\cdots, {\tilde \alpha_P}, 
{\tilde \lambda_1}, \cdots, {\tilde \lambda_P }, y \Bigr)
\\& \times e^{ i \nu^{1/4} \sum_{j=1}^N (\alpha_j - \alpha_0^\eps) x 
+ i \nu^{1/2} \sum_{j=1}^N  [\lambda_i(\alpha_j) - \lambda_i^\eps] t + \nu^{1/2} \sum_{j=1}^N [\lambda_r(\alpha) - \lambda_r^\eps] t}
\\& \times e^{i  \nu^{\gamma} \sum_{j=1}^P \tilde \alpha_j  x
+ \nu^{2 \gamma} \sum_{j=1}^P \tilde \lambda_j t } 
\, d\alpha_1 \, \cdots \, d\alpha_N \, d\tilde \alpha_1 \cdots \, d\tilde \alpha_P \, d\tilde \lambda_1 \, \cdots \, d\tilde \lambda_P,
\end{split} \eeq
where
$$
\alpha_1 =\alpha -\sum_{k=2}^N \alpha_k.
$$
We then have
\beq \begin{split}  \label{generalform2}
\partial_\alpha \widehat V(t, \alpha,y) =&\int_{\mathbb{R}^{N-1}\times \rit^P \times \rit^P }
\Bigl[ \nu^{-\beta} \partial_{\alpha_1} \phi + i \nu^{1/2} \partial_\alpha \lambda_i(\alpha_1) t 
+ \nu^{1/2} \partial_\alpha \lambda_r(\alpha_1) t \Bigr]
\\& \times e^{ i \nu^{1/4} \sum_{j=1}^N (\alpha_j - \alpha_0^\eps) x 
+  i \nu^{1/2} \sum_{j=1}^N  [\lambda_i(\alpha_j) - \lambda_i^\eps] t + \nu^{1/2} \sum_{j=1}^N [\lambda_r(\alpha) - \lambda_r^\eps] t}
\\& \times e^{i  \nu^{\gamma} \sum_{j=1}^P \tilde \alpha_j  x
+ \nu^{2 \gamma} \sum_{j=1}^P \tilde \lambda_j t } 
\, d\alpha_1 \, \cdots \, d\alpha_N \, d\tilde \alpha_1 \cdots \, d\tilde \alpha_P \, d\tilde \lambda_1 \, \cdots \, d\tilde \lambda_P,
\end{split} \eeq
We will consider times such that $\nu^{1/2} t \lesssim \log \nu^{-1}$, thus, we obtain, by iteration, that 
\beq \label{generalform3}
| \partial_\alpha^p \widehat V(t, \alpha,y) | \lesssim \nu^{- p \beta}  ,
\eeq
which leads to (\ref{generalform4}).
As $y$ is a simple parameter in (\ref{generalform2}), if we have bounds on $\phi$ which depend on $y$, then we get the same $y$ dependency on the bounds of $f$ in (\ref{generalform4}).
\end{proof}


\subsection{Proof of Theorem \ref{theogeneral}}


\subsubsection*{Lighting up the instability}

The first step is to "turn on" the instability. Let $\theta(t)$ be a smooth function which equals $0$ for $t \le 0$ and $1$ for $t \ge 0$.
We define $v^\nu - U$ through its stream function $\psi(t)$, and choose, for $0 \le t \le 1$,
$$
\psi(t) = \nu^N \theta(t) \Psi_{lin}(0) + \hbox{c.c.},
$$
where $\Psi_{lin}$ is defined in (\ref{local}). 
We then define $f^\nu$ for $0 \le t \le 1$ to be the force which is needed for  $v^\nu$ to be a solution of Navier Stokes equations
with forcing term $f^\nu$, namely we choose $f^\nu = NS(v^\nu)$ and note that $f^\nu$ is of order $\nu^N$.
By construction $\psi(1) = \nu^N \Psi_{lin}(0)$. We define $t' = t - 1$ and go on with the construction of the approximate
solution for $t' \ge 0$. For simplicity we drop the "prime" index of $t'$, and called it again $t$.

\subsubsection*{Construction of the instability}

Let us now describe the construction of the approximate solution. We proceed by iteration, starting from $\Psi_{lin} + \bar \Psi_{lin}$,
with a corresponding vorticity $\omega_{lin} + \bar \omega_{lin}$. This vorticity then "interacts" with the background velocity field through
$\tilde L_{NS}^1$ and also through the transport term $Q(\Psi_{lin},\Psi_{lin})$. Note that interaction with the background
do not change the values of $m$ and $n$, but creates polynomials in $t$, namely $m = \pm 1$, $n = 1$ and $p \ge 1$ when we inverse
Orr Sommerfeld equations.
Quadratic interactions however lead to $m= -2$, $0$ and $2$, to $n = 2$ and $p = 0$ when we inverse Orr Sommerfeld operator.

Let us first describe the nonlinear interaction of two functions $V$ and $V'$ of the form (\ref{form1})
through the transport term $(V_1 \cdot \nabla) V_2$. 
Let $\phi$ and $\phi'$, $N$ and $N'$ be the corresponding parameters of $V$ and $V'$.
First we observe that the vorticity $\omega$ corresponding to $V$ is also of the form (\ref{form1}), with corresponding
function $\phi_\omega$. The $X^q_\omega$ norms, for $q \ge 0$, of $\phi_\omega$ are bounded by the $X^{q,0}$ norms of
$\phi$.
The interaction term $(V' \cdot \nabla) \omega$ is also of the form (\ref{form1}), with now an integral over $\rit^{N+N'}$.
The $X^q_\omega$ norms of this term are bounds by the $X^{q,0}$ norms of $\phi$ and $\phi'$,
 together with a gain of a factor $\nu^{1/4}$ thanks to (\ref{u5}).
The inversion of Orr Sommerfeld equations leads to a new velocity field  $V''$, again of the same form, with a loss
of a factor $\nu^{-1/2}$ in the $X^{q,0}$ norms, as indicated by Lemma \ref{away}.
The $X^{q,0}$ norms of $V''$ are thus bounded in terms of the $X^{q,0}$ norms of $\phi$ and $\phi'$, up to the
loss of a factor $\nu^{-1/4}$.
Note that $n \ge 1$ and $n' \ge 1$, thus $n + n' \ge 2$, namely we are away from the spectrum, and no polynomial
term appears.

Let us turn to the interactions with the background. Let $\omega$ be the vorticity of $V$. 
We observe that $\tilde L_{NS}^1 (\omega)$ is of the same form. Inverting Orr Sommerfeld
again leads to the same form, with also polynomial factors if $n = 1$. 
Let us turn to bounds. The first two terms of $\tilde L_{NS}^1$ are of order $\nu^{1/4} \eps \nu^{-2\gamma}$,
recalling that we gain a factor $\nu^{1/4}$ in (\ref{u5}), and using Lemma \ref{gain}. We lose a factor $\nu^{-1/2}$ by inverting
Orr Sommerfeld, thus these two terms leads to a contribution of size 
$\eps \nu^{-1/4 - 2 \gamma} = \nu^\gamma$ for some small positive $\gamma$ since $\delta > 3/4$ and
$\gamma$ may be chosen arbitrarily close to $1/4$.
The other terms of $\tilde L_{NS}^1$ are smaller thanks to the $\eps$ factor in front of them.

Moreover, combining  (\ref{generalform4}) and (\ref{rewriteNS2}), we obtain that the error term ${\cal E}$ is smaller than
$O(\nu^N)$.

\subsubsection*{Conclusion}

It remains to discuss the sizes of the various terms. Using Lemma \ref{boundpsin},
the leading term is $\Psi_{lin}$ which grows like $\nu^N \exp(\nu^{1/2} \Re \lambda t) / \sqrt{ \nu^{1/2} t}$. 
The next term comes from the quadratic interaction
and grows like $\nu^{2N} \exp(2 \nu^{1/2} \Re \lambda t)/ \nu^{1/4} (\nu^{1/2} t)$. Higher  interactions, of order $P$, grow like 
$\nu^{PN} \exp(\nu^{1/2} P \Re \lambda t)/ \nu^{P/4} (\nu^{1/2} t)^{P/2}$, again
using Lemma \ref{boundpsin}. Note the loss of a factor $\nu^{1/4}$ at each step.

 The end of the proof follows the lines of \cite{Gr1}: we define $T^\nu$ such that
$$
\nu^N {e^{\nu^{1/2} \Re \lambda T^\nu} \over \sqrt{\nu^{1/2} T^\nu}} = \nu^{1/4 + \theta} .
$$
For $t = T^\nu$ we observe that the quadratic term is a factor $\nu^\theta$ smaller than the linear one,
and more generally the term of order $P$ is a factor $\nu^{P \theta}$ smaller than the linear one.
In particular the error term is of order $O(\nu^N)$ provided the approximate solution contains enough terms.
At $T^\nu$, the approximate solution displays the scales in $y$ of $\Psi_{lin}$.


\section{"Fast" instabilities \label{fast}}


We now turn to the case where $(U_1^0(0,0,0,y),0)$ is unstable with respect to Euler equations.
In all this section we do not rescale $\alpha$ by a factor $\nu^{1/4}$.

\begin{theo} \label{theogeneralfast}
Let us assume that $(U_1^0(0,0,0,y),0)$ is a spectrally unstable shear layer for Euler equations, namely that there exists a
solution $(\alpha_{Ray},c_{Ray},\psi_{Ray})$ to the Rayleigh equation (\ref{Rayleigh}) with $\Im c_{Ray} > 0$,
$\alpha_{Ray} \ne 0$,  $\psi_{Ray}(y) \not\equiv 0$ and $\psi_{Ray}(0) = 0$.
 
Assume that $\eps(\nu) \lesssim \nu^\gamma$ for some $\gamma > 0$. Then, for any arbitrarily small  positive $\theta$,  
for any arbitrarily large $N$ and arbitrarily large $s$, there exists a solution $v^\nu$ of Navier Stokes equations with forcing term $f^\nu$,
 and a time $T^\nu$, such that, for $\nu$ small enough,
\beq \label{theogeneral1-2}
\| v^\nu(0,\cdot,\cdot) - u^{\eps(\nu)}(0,\cdot,\cdot) \|_{H^s} \le \nu^N,
\eeq
\beq \label{theogeneral2-2}
\| f^\nu \|_{L^\infty([0,T^\nu],H^s)} \le \nu^N
\eeq
and
\beq \label{theogeneral3-2}
\| u^\nu(T^\nu,\cdot,\cdot) - u^{\eps(\nu)}(T^\nu,\cdot,\cdot) \|_{L^\infty} \ge  \nu^\theta ,
\eeq
with 
$$
T^\nu \sim C_0 \log \nu^{-1}
$$
for some positive constant $C_0$.
\end{theo}

\begin{proof}
We split the proof into two parts. First we construct an approximate eigenmode $\psi_{Orr}$ for Orr Sommerfeld equation, starting from
$\psi_{Ray}$. Then we construct $v^\nu$ using $\psi_{Orr}$.

The first step, namely the construction of $\psi_{Orr}$, has already been described in \cite{GN3}. We just recall its main lines here,
since it will be used in the next section. The second step 
is similar to section \ref{general}, and even easier, since there is no $\nu^{1/4}$ and no $\nu^{1/2}$ in the exponential factors.

We now recall the construction of an approximate solution for Orr Sommerfeld equation, 
starting from an unstable mode of Rayleigh equations.
Note that, $(c_{Ray},\phi_{Ray})$ is not an approximate solution of Orr Sommerfeld equation, since it does not fit the boundary condition 
$\partial_y \phi_{Ray}(0) = 0$ (else, as $\phi_{Ray}(0) = 0$, this would mean that $\phi_{Ray}$ identically vanishes).
We therefore have to add a "boundary layer" to $\phi_{Ray}$ in order to ensure this boundary layer condition.
We thus look for approximate solutions $\phi_{Orr}$ of Orr Sommerfeld equation of the form
\beq \label{exp1}
\phi_{Orr}(y) = \sum_{n \ge 0}^N \nu^{n/2} \phi_n^{int}(y) + \sum_{n \ge 1}^N \nu^{n/2} \phi_n^{bl}(\nu^{-1/2} y)  
\eeq
together with the following expansion of the eigenvalue
\beq \label{exp1bis}
c_{Orr} = \sum_{n \ge 0}^N \nu^{n/2} c_{n},
\eeq
where $N$ is a large integer.
In this expansion $\phi_n^{int}$ refers to an "interior" behaviour, and $\phi_n^{bl}$ to a "boundary layer" behaviour. 
We start with $\phi_0^{int} = \phi_{Ray}$ and $\phi_0^{bl} = 0$. 
The first boundary layer profile satisfies
\beq \label{hie2}
{ \partial_Y^4 \phi_1^{bl}  \over i \alpha_0}  = (U_s(0) - c_0)  \partial_Y^2 \phi_1^{bl} ,
\eeq
together with the boundary condition
\beq \label{hie3}
\partial_Y \phi_1^{bl}(0) = - \partial_y \phi_0^{int}(0),
\eeq
Hence
$$
\phi_1^{bl}( \nu^{-1/2}  y) 
= - \partial_{y} \phi_0^{int}(0) \mu^{-1} \Bigl( 1   - e^{- \nu^{-1/2} \mu y} \Bigr), \qquad \mu = (- i \alpha_0 c_0) ^{1/2},
$$
where we have chosen the square root in such a way that $\Re \mu > 0$.

The various functions $\phi_n^{int}$ and $\phi_n^{bl}$ may then be constructed by iteration. 
The output of this first phase is an approximate mode $\psi_{Orr}$ of Orr Sommerfeld, with associated approximate eigenvalue $c_{Orr}$.
Let
$$
\psi_{lin}(t,x,y) = e^{i \alpha_0 (x - c_{Orr} t)} \psi_{Orr}(y) .
$$
Then the associated velocity $v_{lin} = \nabla^\perp \psi_{lin}$ satisfies
$$
\partial_t v_{lin} + L_{NS}[U] v_{lin} = \nu^N E_{lin},
$$
where $U = (U_1^0(0,0,0,y),0)$,
 $N$ may be chosen arbitrarily large and $E_{lin}$ is some bounded error term.
A similar construction may be done for $(U_1^\eps(0,0,0,y),0)$ provided $\eps$ is small enough.

Let us now compute the boundary layer term of the velocity field. Using
$$
(u^{bl}, v^{bl} ) = \sqrt{\nu} \nabla^\perp \Bigl( e^{i \alpha_0 (x - c_{Orr} t)} \phi^{bl}_1(\nu^{-1/2} y)  + \hbox{c.c.} + O(\sqrt{\nu}) \Bigr),
$$
we see that $u^{bl}$ and $v^{bl}$ have the following asymptotic expansions
$$ 
u^{bl}(t,x,y) = -  e^{i \alpha_0 (x - c_{Orr} t)}  e^{- \nu^{-1/2} \mu y} \partial_{y} \phi_0^{int}(0) + \hbox{c.c.} + O(\nu^{1/2})
$$
and
 $$
 v^{bl}(t,x,y) = - i \alpha \nu^{1/2}  e^{i \alpha_0 (x - c_{Orr} t)} \mu^{-1}
 \Bigl( 1   - e^{- \nu^{-1/2} \mu y} \Bigr)   \partial_{y} \phi_0^{int}(0) 
  + \hbox{c.c.} + O(\nu).
$$
We note that the horizontal speed in the boundary layer is of order $O(1)$ and exactly compensates the interior horizontal component
of the velocity at the boundary. The vertical speed is of size $O(\nu^{1/2})$ and vanishes as $\nu \to 0$.
Both $u^{bl}$ and $v^{bl}$ are periodic in $x$, with period $2 \pi / \alpha_0$, of order $O(1)$.

The construction of an approximate solution then follows the lines of the proof of Theorem \ref{general},
with no $\eps^{1/4}$ or $\eps^{1/2}$ factors.
Namely, the leading term is $\Psi_{lin}$  grows like $\nu^N \exp(\Re \lambda t) / \sqrt{ t}$. 
The next term comes from the quadratic interaction
and grows like $\nu^{2N} \exp(2 \Re \lambda t)/ t$. Higher  interactions, of order $P$, grow like 
$\nu^{PN} \exp( P \Re \lambda t)/  t^{P/2}$. 
The end of the proof follows the lines of \cite{Gr1}: we define $T^\nu$ such that
$$
\nu^N {e^{\Re \lambda T^\nu} \over \sqrt{T^\nu}} = \nu^{\theta} .
$$
For $t = T^\nu - \tau$, we observe that the linear term dominates the sum of all the other
ones provided $\tau$ is large enough, and that the error term is of order $O(\nu^N)$, which gives the desired result.
\end{proof}


\section{Proof of the Theorems}


In this section we do not rescale $\alpha$ by a factor $\nu^{1/4}$.


\subsection{Proof of Theorem \ref{theofirst2} }


Note that the profile $(U_s(y),0)$ is independent of the slow variables $T$, $X$ and $Y$,  and is thus of the 
form (\ref{ueps}) with $\eps = 0$.
Following \cite{Bian4,Zhang}, this profile is spectrally unstable if $U_s'(0) \ne 0$, with unstable modes described in section \ref{linear}..
We then apply  Theorem \ref{theogeneral}, which leads to Theorem \ref{theofirst2}.


\subsection{Proof of Theorem \ref{theomain}}


This Theorem describes a double instability: first the main flow is unstable, with an unstable mode which creates a
sublayer of size $\nu^{1/2}$. Then this sublayer becomes itself unstable.
The first instability is "fast" and the second is "slow".


\subsubsection{Description of the first instability \label{firstinsta}}


Let us start with the case of holomorphic shear layer profiles.
The first instability is constructed in details and justified in \cite{GN2}, and sketched in section \ref{fast}.
The main theorem of \cite{GN2} exactly
provides the existence of a solution $V_1^\nu$ of Navier Stokes equation with forcing term $F^\nu$
on a time interval $[0,T_1^\nu]$, satisfying (\ref{d1})-(\ref{d4}). In particular, at time $T_1^\nu$,
$$
\| V_1^\nu(T_1^\nu,\cdot,\cdot) - V_0^\nu(T_1^\nu,\cdot,\cdot) \|_{L^\infty} \ge \sigma 
$$
for some $\sigma$ independent on $\nu$.
By construction $V_1^\nu(T_1^\nu,\cdot,\cdot)$ displays the sizes $O(1)$ and $O(\nu^{1/2})$ in $y$ and is of the form
$$
V_1^\nu(T_1^\nu + t,x,y) = V_1^{int,\nu}(t,x,y) + V_1^{bl,\nu}(t,x,\nu^{-1/2} y) .
$$
In the case of smooth but non holomorphic profiles $U_s$, Theorem \ref{theogeneralfast} provides
a similar solution $V_1^\nu$, but which only satisfies (\ref{d3bis}).


\subsubsection{Description of the secondary instability }


Let us first investigate the case of holomorphic profiles.
We will build a secondary instability in the sublayer of  size $\nu^{1/2}$ of $V_1^\nu(T_1^\nu,\cdot,\cdot)$.
For this, we  rescale $V_1^\nu$  in order to focus on this sublayer $O(\nu^{1/2})$.
More precisely we define
$$
u(t,x,y) = V_1^{int,\nu} \Bigl( T_1^\nu + \nu^{1/2} t, \nu^{1/2} x, \nu^{1/2} y \Bigr)
+ V_1^{bl,\nu} \Bigl( T_1^\nu + \nu^{1/2} t, \nu^{1/2} x, y \Bigr) .
$$
After this rescaling, $u(t,x,y)$ has scales of order $O(1)$ and $O(\nu^{-1/2})$ in $y$, and of order $O(\nu^{-1/2})$ in $x$ and $t$.
It is also of the form (\ref{ueps}) with $\eps = \nu^{1/2}$.
With this rescaling, the viscosity of Navier Stokes equations changes from $\nu$ to $\nu_1 = \sqrt{\nu}$.
Moreover, $u(0,0,y)$ is explicit in $y$, and of the form
$$
V_1^{bl,\nu}(0,0,y) =  - V_1^{int,\nu}(0,0,0) \Bigl( 1 - e^{-\mu y},0 \Bigr)+ O(\nu^{1/2}).
$$
We can choose the point $x = 0$ in such a way that $V_1^{int,\nu}(t,0,0) \ne 0$.
Then $\partial_y V_1^{bl,\nu}(0,0,0) \ne 0$. 
As the viscosity $\nu_1 = \nu^{1/2}$ goes to $0$, this shear layer itself becomes linearly unstable.

Let $\alpha = \alpha_0 \nu_1^{1/4}$ be an unstable wave number for this shear layer profile. 
Let $\lambda(\alpha,\nu) = O(\nu_1^{1/4})$ be an unstable eigenvalue
with corresponding eigenmode $\psi_{lin}$. Then $\psi_{lin}$ exhibits the scales $O(1)$, $O(\nu_1^{-1/4})$ and $O(\nu_1^{1/4})$,
namely $O(\nu^{1/2})$, $O(\nu^{3/8})$ and $O(\nu^{5/8})$ in the original variables.

We now apply Theorem \ref{theogeneral}, with $\eps = \nu^{1/2}$ and $\delta = 1$ since the viscosity is now $\nu_1 = \nu^{1/2}$,
which ends the proof.

Let us turn to the case of $C^\infty$ profiles.
The first nonlinear instability described in the previous paragraph only reaches $O(\nu^\delta)$ for an arbitrarily small $\theta$,
thus, $V_1^{bl,\nu}$ has only a magnitude $O(\nu^\delta)$ and not $O(1)$ as previously.
Let us consider Orr Sommerfeld with a shear layer $\tilde U_s(y)$ of the form $\tilde U_s(y) = \nu^\theta U_s(y)$ where $U_s(y)$ 
is a given function.
Let $c(\alpha,\nu)$ be the dispersion relation corresponding to $U_s$.
Then the dispersion relation $\widetilde c(\alpha,\nu)$ corresponding to $\tilde U_s$ is
$$
\widetilde c(\alpha,\nu) = \nu^\theta c (\alpha,\nu^{1 - \theta}) .
$$
In particular, provided $\theta$ is small enough, $\tilde U_s$ is also unstable. The most unstable eigenvalue $\lambda$ has
a growth of order $\nu^{1/2 +  \theta / 2}$.
The proof of Theorem \ref{theogeneral} can be extended to this case. The instability will develop with times of order
$T_\nu = \nu^{-1/2 - \theta/2} \log \nu^{-1}$, namely slightly slower.


\subsection{Proof of Theorem \ref{theomain2}}


The proof of Theorem \ref{theomain2} is close to that of Theorem \ref{theomain}. 
It is a double instability. First Prandtl's layer is nonlinear unstable. This first instability creates a sublayer of size $\nu^{1/4}$
which becomes unstable, leading to a secondary instability.
The first instability is "fast" and the second is "slow".


\subsubsection{First instability}


We first rescale $u$ and introduce
$$
u_1^{\nu}(t,x,y) = u^{Euler,\nu}(t_0 + \nu^{1/2} t, x_0 + \nu^{1/2} x, \nu^{1/2} y) 
+ u^{Prandtl,\nu}(t_0 + \nu^{1/2} t, x_0 + \nu^{1/2} x,  y) 
$$
which is of the form (\ref{ueps}) with $\eps = \sqrt{\nu}$.
By assumption, $u_1^0(0,0,0,y)$ is an unstable shear layer for Euler equations. Thus, there exists an unstable eigenmode
$\psi_1^0$, with wave number $\alpha_1^0$ and an eigenvalue $c_1^0$. Note that both $\alpha_1^0$ and $\Im c_1^0$ are of order $O(1)$.

We thus may apply Theorem \ref{theogeneralfast} to prove that the instability grows nonlinear and reach a size $O(\nu^{\theta})$.
In the case of analytic initial data, using technics developed in \cite{GN2},  we can even take $\theta = 0$ and
prove that the instability reaches a size of order $O(1)$.


\subsubsection{Second instability}


The secondary instability is then similar to that of Theorem \ref{theomain}.


\subsection{Proof of Theorem \ref{theomain4}}


The first step is to rescale space and time and to introduce
$$
t_1 = {t \over \sqrt{\nu}}, \quad 
x_1 = {x \over \sqrt{\nu}}, \quad 
y_1 = {y \over \sqrt{\nu}},
$$
together with
$$
u_1(t_1,x_1,y_1) = u(\sqrt{\nu} t_1, \sqrt{\nu} x_1, \sqrt{\nu} y_1) .
$$ 
The viscosity in the rescaled variables is now
$$
\nu_1 = \sqrt{\nu}.
$$
and $u_1$ is of the form (\ref{ueps}) with $\eps = \sqrt{\nu} = \nu_1$.
In these new variables, $u_1$ has two scales in $y$, namely $O(\nu^{-1/2})$ (Euler scale) and $O(1)$ (Prandtl's boundary layer scale).

Using the results of \cite{Bian4,Zhang}, we know that the $O(1)$ layer is unstable for horizontal wavenumbers $\alpha_1$ of order $O(\nu_1^{1/4})$,
with corresponding eigenvalues $c_\nu$ of order $O(\nu_1^{1/4})$. 
The typical time of growth of such instabilities is of order $O(\nu_1^{-1/2})$.
The corresponding eigenvector $\psi$ exhibits three vertical sizes: $O(1)$, $O(\nu_1^{1/4})$ and $O(\nu_1^{-1/4})$, or
equivalently, in the initial variable $y$, $O(\nu^{1/2})$, $O(\nu^{5/8})$ and $O(\nu^{3/8})$.
These instabilities travel with a group velocity of order $\nu_1^{1/4}$,  over lengths of order $\nu_1^{-1/4} \log \nu_1$, namely
$\nu^{5/8} \log \nu \ll 1$ in the original $x$ variable.
Theorem \ref{theomain4} is then an application of Theorem \ref{theogeneral}.


\subsubsection*{Acknowledgments}  D. Bian is supported by NSFC under the contract 12271032.



\end{document}